\documentclass[a4paper,11pt]{amsart}
\usepackage[height=24.9cm,a4paper,hmargin={3cm,3cm},top=2.8cm]{geometry}
\usepackage{amssymb,mathtools}
\usepackage{enumerate}
\usepackage[T1]{fontenc}
\usepackage{lmodern}
\usepackage[english]{babel}
\usepackage{todonotes}

\usepackage{xcolor}
\usepackage{hyperref}
\newtheorem{theorem}{Theorem}[section]
\newtheorem{lemma}[theorem]{Lemma}
\newtheorem{proposition}[theorem]{Proposition}
\newtheorem{corollary}[theorem]{Corollary}

\theoremstyle{definition}
\newtheorem{definition}[theorem]{Definition}
\newtheorem{model}{Hypothesis}
\theoremstyle{remark}
\newtheorem{remark}[theorem]{Remark}
\newtheorem{example}[theorem]{Example}
\newenvironment{hypothesis}[1]{
	\renewcommand\themodel{#1}
	\model
}{\endmodel}
\numberwithin{equation}{section}
\newcommand{\RR}{\mathbb{R}}
\newcommand{\CC}{\mathbb{C}}
\newcommand{\NN}{\mathbb{N}}
\newcommand{\ZZ}{\mathbb{Z}}
\newcommand{\cE}{\mathcal{E}}
\newcommand{\cJ}{\mathcal{J}}
\newcommand{\cI}{\mathcal{I}}
\newcommand{\cB}{\mathcal{B}}
\newcommand{\cU}{\mathcal{U}}
\newcommand{\cS}{\mathcal{S}}
\newcommand{\cF}{\mathcal{F}}
\newcommand{\cD}{\mathcal{D}}
\newcommand{\cT}{\mathcal{T}}

\newcommand{\cH}{\mathcal{H}}
\newcommand{\cG}{\mathcal{G}}
\newcommand{\cK}{\mathcal{K}}
\newcommand{\cR}{\mathcal{R}}
\newcommand{\cN}{\mathcal{N}}
\newcommand{\ii}{\mathrm{i}}
\newcommand{\eps}{\varepsilon}
\newcommand{\euler}{\mathrm{e}}
\newcommand{\w}{\mathrm{w}}
\renewcommand{\aa}{\mathfrak{a}}
\newcommand{\hh}{\mathfrak{h}}
\newcommand{\vv}{\mathfrak{v}}
\newcommand{\indic}{\mathbf{1}}

\DeclareMathOperator{\Ran}{Ran}
\DeclareMathOperator{\diam}{diam}
\DeclareMathOperator{\Span}{span}
\newcommand*\Diff[1]{\mathop{}\!\mathrm{d}#1}
\let\Re\relax
\DeclareMathOperator{\Re}{Re}
\let\Im\relax
\DeclareMathOperator{\Im}{Im}
\DeclareMathOperator{\Id}{Id}
\newcommand{\cOp}{B}
\DeclarePairedDelimiter{\abs}{|}{|}
\DeclarePairedDelimiter{\norm}{\lVert}{\rVert}
\newcommand{\bes}{\begin{equation*}}
\newcommand{\ees}{\end{equation*}}
\newcommand{\be}{\begin{equation}}
\newcommand{\ee}{\end{equation}}
\newcommand{\eqs}[1]{\begin{align*}#1\end{align*}}

\definecolor{darkred}{rgb}{0.5,0,0}
\definecolor{darkgreen}{rgb}{0,0.5,0}
\definecolor{darkblue}{rgb}{0,0,0.5}
\title[Control problem via partial harmonic oscillators]
{Control problem for quadratic parabolic differential equations with sparse sensor sets of finite volume or anisotropically decaying density}
\date{\today}
\hypersetup{
	pdftitle={Control theory for quadratic differential operators with sensor sets of decaying density via partial harmonic oscillators},
	pdfauthor={Alexander Dicke, Albrecht Seelmann, Ivan Veseli\'c},
	colorlinks,
	linkcolor=darkblue,
	filecolor=darkgreen,
	urlcolor=darkred,
	citecolor=darkblue,
}
\author[A.~Dicke]{Alexander Dicke}
\author[A.~Seelmann]{Albrecht Seelmann}
\author[I.~Veseli\'c]{Ivan Veseli\'c}
\address[A.D.]{
Dortmund, Germany
}
\email{adicke.math@gmail.com}
\address[A.S., I.V.]{
	Technische Univer\-si\-t\"at Dortmund,
	Germany
}
\urladdr{\url{https://www.mathematik.tu-dortmund.de/lsix/research/analysis/}}
\email{\{aseelman,iveselic\}@mathematik.tu-dortmund.de}
\begin{document}
%
%
\begin{abstract}
	We prove observability and null-controllability for quadratic parabolic differential equations.
	The sensor set is allowed to be sparse and have finite volume if the generator has trivial singular space $S$.
	In the case of generators with singular space $S \neq \{ 0 \}$ the sensor set is permitted to decay in directions determined by $S$.
	The proof is based on dissipation estimates for the quadratic differential operator with
	respect to spectral projections of partial harmonic oscillators and corresponding uncertainty relations.
\end{abstract}
\maketitle
%
%
\section{Introduction}

We treat the control problem for quadratic parabolic differential equations, including hypoelliptic ones.
For a wide class of such equations we show that starting from a given initial state it is possible to drive the
solution to zero in a given time $T\in (0,\infty)$ by steering it from sensor sets that are sparse at infinity and
may have finite measure.
Typical examples of such sensor sets $\omega$ satisfy
for some $\rho, \gamma>0$ and $a \in (0,1)$
the geometric condition
\be\label{eq:decay-bound}
	\gamma^{1+|x|^a}
	\leq
	\frac{|\omega \cap B(x,\rho)|}{|B(x,\rho)|}
	\quad \text{for all}\quad x \in\RR^d.
\ee
Here $B(x,\rho)$ is the ball of radius $\rho > 0$ centered at $x\in\RR^d$.
Such $\omega $ have finite measure, if complemented by a similar upper bound, more specifically
if for some $0< a^- \leq a^+<1$ we have
\be\label{eq:decay-twosided}
	\gamma^{1+|x|^{a^-}}\leq
	\frac{|\omega \cap B(x,\rho)|}{|B(x,\rho)|}
	\leq
	\gamma^{1+|x|^{a^+}}
	\quad \text{for all \quad $x \in\RR^d$}.
\ee
A particular instance of such a set is
\bes
	\omega
	=
	\bigcup_{k\in\ZZ^d} B\Bigl(k,2^{-1-\sqrt{|k|}}\Bigr).
\ees
The prime example of a quadratic differential operator such that the corresponding parabolic
equation exhibits observability from a sensor set as described above is the \emph{harmonic oscillator Hamiltonian}.
This fact extends to semigroups associated to a class of quadratic differential operators
comparable to the harmonic oscillator in a sense that will be made precise below.
A more general class of operators related to so-called \emph{partial harmonic oscillators}, also
defined below, give rise to parabolic equations that are observable from certain anisotropic cousins of
sensor sets $\omega$ as in \eqref{eq:decay-twosided}.
A precise statement is formulated in Theorem~\ref{thm:observability} below, and a wider class of admissible sensor sets can be
found in Corollary~\ref{cor:specIneq:genSets}.

As just mentioned, the fundamental and most basic example for which our observability inequalities hold
is the harmonic oscillator. For this special case, these inequalities are stated and proved in
our short companion paper \cite{DickeSV-21} and we refer readers who just want to grasp our results in the simplest case to this note.
In contrast, those who want to see the full scope of our methods should stick with the paper at hand, where we treat a broad class of models beyond the harmonic oscillator.
Some of the ideas of the proofs in \cite{DickeSV-21} reappear here in a generalized variant,
giving rise to parallels between \cite{DickeSV-21} and Section \ref{sec:spectral-inequality} below.

Our results improve and generalize several recently established criteria, such as \cite{BeauchardJPS-21,MartinPS-20-arxiv,DickeSV-21,Alphonse-20},
for observability and null-control\-lability, especially for semigroup generators comparable to the harmonic oscillator.
While we cannot prove at this point that our sufficiency criteria are also necessary in this case,
bounded sensor sets do not yield a spectral inequality (a certain type of uncertainty relation,
which is a crucial step in proving observability by the established Lebeau-Robbiano method)
see \cite{Miller-08} or \cite[Example~2.5]{DickeSV-21}.

The main body of results in the present paper concern controllability for semigroups generated by a differential operator
comparable to an \emph{anisotropic} Schr\"odinger operator.
Here, the potential of the Schr\"odinger operator is growing unboundedly in some coordinate directions while being
bounded or constant in others.
A prototypical example on $\RR^2$ is the \emph{partial harmonic oscillator}
\be\label{eq:2d-partial-harmonic-oscillator}
	-\frac{\partial^2}{\partial x^2} -\frac{\partial^2}{\partial y^2} + x^2.
\ee
Naturally, the anisotropy of the potential is reflected in the criteria for a sensor set to yield null-controllability.
In the case of \eqref{eq:2d-partial-harmonic-oscillator}, the sensor set is allowed to decay in the $x$-direction in the manner described in \eqref{eq:decay-bound},
while in the $y$-direction decay is not permitted.
In fact, in direction of the $y$-variable the sensor set should be \emph{thick} in the sense of formula \eqref{eq:thick} below.

The control problems we are considering are defined on unbounded domains. This framework has recently attracted considerable attention
due to applications in kinetic theory.
There the geometric domain of the differential equation is typically an unbounded subset of the phase space due to absence of restrictions on the velocity coordinates.
Our results are general enough to cover, among others, certain types of
generalized Ornstein-Uhlenbeck operators, especially those appearing in Kolmogorov and Kramers-Fokker-Planck equations.
It is, however, worth to mention that similar results for non-quadratic operators
such as (anisotropic) Shubin operators and Schr\"odinger operators with
power growth potentials have recently been obtained in \cite{Martin-22,DickeSV-22b}.

We summarize now the new findings and methodological advancements
in the present paper and take this opportunity to explain its structure.
To begin with, in Section~\ref{sec:previous-results} below we recall the well-established Lebeau-Robbiano method, which combines a so-called dissipation estimate
with a suitable uncertainty relation (resp.~spectral inequality) to conclude observability.
Moreover, in that section we review relevant previous results.

Section~\ref{sec:main_results} presents the three major contributions of the present paper (in a somewhat simplified form).
The first one is the dissipation estimate in Theorem~\ref{thm:dissipation} for general quadratic differential operators
exhibiting a singular space with product structure, cf.~Corollary \ref{cor:rotated-dissipation} for our most general result in this respect.
The second is an anisotropic spectral inequality for partial harmonic oscillators formulated in Theorem~\ref{thm:ucp};
Section~\ref{sec:spectral-inequality} contains the proof of a generalized version, Theorem~\ref{thm:gen}.
The last and final one is an observability result from sensor sets with decaying density
or even finite measure, see Theorem~\ref{thm:observability} and its extension Theorem~\ref{thm:observability-sharp-control-costs}.
These generalize the findings of our companion paper \cite{DickeSV-21},
where the first example of observability from finite measure sets for a quadratic operator on $\RR^d$ was exhibited.

The way how partial harmonic oscillators enter the picture becomes clear in Section~\ref{sec:dissipation}.
The singular space of a quadratic differential operator $A$ allows us to associate to it a particular
partial harmonic oscillator $H_\cI$ in such a way that the semigroup generated by $A$ satisfies a
dissipation estimate with respect to the spectral projections of $H_\cI$.
Theorem~\ref{thm:dissipation}
unifies, interpolates, and generalizes the dissipation estimates derived earlier in \cite{BeauchardPS-18,BeauchardJPS-21, Alphonse-20,MartinPS-20-arxiv}.
Section~\ref{sec:dissipation} provides the proof of our dissipation estimate as well as some extensions.

While our proof of the dissipation estimate is based, as in the previously mentioned papers, on anisotropic Gelfand-Shilov smoothing estimates
we provide a considerably streamlined derivation compared to earlier ones.

Applications and extensions of our results are presented in Section~\ref{sec:applications-extensions}:
We show that partial harmonic oscillators allow a more explicit
treatment with quantitative bounds on the control costs.
We also consider Shubin as well as generalized Ornstein-Uhlenbeck operators that fit into our framework;
amongst others, this includes the Kolmogorov and the Kramers-Fokker-Planck equations.

The exposition in this paper aims at accessibility for non-experts and for completeness sake
includes some arguments spelled out in the literature before.
For the same reason we provide an appendix containing a proper definition of partial harmonic oscillators.
Moreover, we there provide dimension reduction arguments based on the tensor structure of these operators.

\subsection*{Acknowledgments}

A.S. is indebted to M.~Egidi for inspiring discussions leading to the proof of Lemma~\ref{lem:Bernstein}.
I.V. would like to thank C. Th\"ale and B.~Gonzales Merino for references to the literature on convex bodies.
A.D. and A.S. have been partially supported by the DFG grant VE 253/10-1 entitled
\emph{Quantitative unique continuation properties of elliptic PDEs with variable 2nd order coefficients and applications in control theory, Anderson localization, and photonics}.
%
%
\section{Model, previous results, and goals}\label{sec:previous-results}

Let $A \colon \cH \supset \cD(A) \to \cH$ be a densely defined, closed operator
on a Hilbert space $\cH$ generating a strongly continuous semigroup
$(\cT(t))_{t\geq 0}$ and let $\cOp \in \mathcal{L}(\cH)$ be a bounded operator.
We consider the abstract Cauchy problem
\be\label{eq:controlled-abstract-Cauchy-problem}
	w'(t)
	=
	A w(t)
	,
	\quad
	v(t) = \cOp w(t)
	,
	\quad
	w(0)
	=
	w_0
	,
\ee
where $t \geq 0$.
The latter is said to be \emph{final-state observable} in time $T > 0$
if there is a constant $C_{\mathrm{obs}} > 0$ such that the observability inequality
\[
	\norm{\cT(T)g}_{\cH}^2
	\leq
	C_{\mathrm{obs}}^2 \int_0^T \norm{\cOp \cT(t)g}_{\cH}^2 \Diff{t}
	\quad \text{holds for all}\quad
	g\in\cH
	.
\]

For $\cH = L^2(\RR^d)$ observability inequalities have been derived for several combinations of operators $A$ and $\cOp$.
The most important one is the case where
$\cOp = \indic_\omega\colon f \mapsto \indic_\omega f$ with a measurable
set $\omega \subset \RR^d$ and where $A$ is some differential operator.
Here the measurable set $\omega \subset \RR^d$ is called \emph{sensor set} and it is a fundamental problem to understand what geometric
conditions on $\omega$ are necessary and/or sufficient for observability of the associated abstract Cauchy problem
\eqref{eq:controlled-abstract-Cauchy-problem}. Given some (sufficient) conditions for the sensor set, we are also interested in the
dependence of the observability constant $C_{\mathrm{obs}}$ on the geometry of $\omega$.

For the Laplacian $A = \Delta$ on $\RR^d$ sharp geometric conditions on $\omega$ were obtained in \cite{EgidiV-18,WangWZZ-19}.
There it is shown that the Cauchy problem of the Laplacian is observable if and only if the sensor set $\omega$ is \emph{thick}, i.e.,
if there are $\gamma, \rho > 0$ such that
\be\label{eq:thick}
	\frac{|\omega \cap B(x,\rho)|}{|B(x,\rho)|}
	\geq \gamma
	\quad\text{for all}\quad x\in\RR^d.
\ee
Associated bounds on the control cost in terms of the parameters $\rho, \gamma$ were
given in \cite{EgidiV-18} as well, and optimized in \cite{NakicTTV-20}.

Condition \eqref{eq:thick} has been shown in \cite{BeauchardJPS-21} to be sufficient also in the case where $A = \Delta - |x|^2$ is the negative harmonic
oscillator on $\RR^d$.
However, thickness of sensor sets is not necessary for this choice of $A$, cf.~\cite{MartinPS-20-arxiv, DickeSV-21}.
In fact, the main result of \cite{DickeSV-21} implies that the Cauchy problem of $A$ is observable for
sensor sets $\omega$ that satisfy the weaker condition \eqref{eq:decay-bound}.

All these results are based on the so-called Lebeau-Robbiano method.
This method combines a \emph{dissipation estimate} with a suitable \emph{spectral inequality}, see \eqref{eq:dissipation} and \eqref{eq:spectral} below,
to derive an observability inequality for the Cauchy problem.
The mentioned spectral inequality is a particular form of a quantitative unique continuation estimate or uncertainty principle for
elements of spectral subspaces of elliptic differential operators.
Here we spell out a variant of the Lebeau-Robbiano method formulated in \cite{BeauchardPS-18};
we also refer the reader to the closely related works \cite{TenenbaumT-11,BeauchardEPS-20,NakicTTV-20,GallaunST-20}.

\begin{theorem}\label{thm:obs_and_control}
	Let $A$ be the generator of a strongly continuous contraction semigroup $(\cT(t))_{t\geq 0}$ on $L^2(\RR^d)$
	and let $\omega \subset \RR^d$
	be a measurable set with positive Lebesgue measure.
	Suppose that there is a family $(P_\lambda)_{\lambda\in [1,\infty)}$
	of orthogonal projections in $L^2(\RR^d)$ such that
	\begin{enumerate}[(i)]
		\item
		there are $d_0, d_1, \gamma_1 > 0$ such that for all $\lambda\geq 1$ and all $g \in L^2(\RR^d)$ we have
		\be\label{eq:spectral}
			\lVert P_\lambda g \rVert_{L^2(\RR^d)}^2
			\leq
			d_0 \euler^{d_1 \lambda^{\gamma_1}}
			\lVert P_\lambda g \rVert_{L^2(\omega)}^2.
		\ee
		\item
		there are $d_2 \geq 1$, $d_3,\gamma_3, t_0 > 0$, and $\gamma_2 > \gamma_1$ such that for all $\lambda\geq 1$, $0 < t < t_0$, and
		all $g \in L^2(\RR^d)$ we have
		\be\label{eq:dissipation}
			\lVert (I-P_\lambda) \cT(t)g \rVert_{L^2(\RR^d)}^2
			\leq
			d_2 \euler^{-d_3 \lambda^{\gamma_2}t^{\gamma_3}}
			\lVert g \rVert_{L^2(\RR^d)}^2.
		\ee
	\end{enumerate}
	Then there is a constant $C > 0$ such that for all $g \in L^2(\RR^d)$ and all $T > 0$ we have the observability estimate
	\bes
		\bigl\lVert \cT(T) g \bigr\rVert_{L^2(\RR^d)}^2
		\leq
		C_{\mathrm{obs}}^2
		\int_0^T  \bigl\lVert \cT(t) g \bigr\rVert_{L^2(\omega)}^2 \Diff{t} ,
	\ees
	with
	\bes
		C_{\mathrm{obs}}^2
		=
		C \exp \left( \frac{C}{T^{\frac{\gamma_1\gamma_3}{\gamma_2-\gamma_1}}}\right)
		.
	\ees
	Here  $C$ depends merely on $d_1,d_2,d_3$, and  $t_0$.
\end{theorem}

\begin{remark}
	(a) The statement of Theorem~\ref{thm:obs_and_control} in \cite{BeauchardPS-18} was originally formulated for open $\omega$.
	However, the proof only requires measurable $\omega$ with positive measure, as observed in \cite{EgidiV-18}.
	
	(b) The dissipation estimate \eqref{eq:dissipation} is here only required to hold for small $t$.
	There have been attempts to sharpen the bound in the observability estimate \cite{GallaunST-20},
	cf.~also \cite{NakicTTV-20}, but this requires the dissipation estimate \eqref{eq:dissipation} to hold for all $t\in (0,T/2]$.
	On the other hand, in \cite{BeauchardEPS-20}, Theorem~\ref{thm:obs_and_control} above has been extended to allow $d_2=d_2(t)$
	with a polynomial blow-up as $t\to 0$.

	(c) In many cases the dissipation bound in \eqref{eq:dissipation} has $\gamma_2=1$, e.g. if $P_\lambda$ is a spectral projection of $A$,
	cf.~the discussion below. Hence, in what follows we are particularly interested in spectral inequalities with $\gamma_1 <1$.
\end{remark}

Note that in the hypotheses of the above theorem the sensor set only appears in the spectral inequality \eqref{eq:spectral}.
Hence, this is the only pivot where the geometric assumptions on the sensor set $\omega$ influence the Lebeau-Robbiano method.
For instance, the spectral inequality with $P_\lambda = \indic_{(-\infty,\lambda]}(-\Delta)$
holds if and only if the set $\omega$ is thick, see \cite{Kacnelson-73,LogvinenkoS-74}.

Let us turn now to consider the dissipation estimate:
It is trivial by functional calculus if $A$ is a self-adjoint operator (for instance $A = \Delta$ or $A = \Delta - |x|^2$) and one chooses
$P_\lambda= \indic_{(-\infty,\lambda]}(-A)$ to be a projection onto a suitable  spectral subspace.
However, a spectral inequality for spectral projectors of $A$ might not be available, or, even worse, the operator $A$ might not be self-adjoint.
In that case it is natural to search for a suitable self-adjoint `comparison' operator $H$
and choose the operators $P_\lambda$ as its spectral projections.
Spectral inequalities for these projections then directly determine the possible sensor sets.

To the best of the authors knowledge, this approach was first implemented in \cite{BeauchardPS-18} for certain
(quadratic) differential operators $A$ with the harmonic oscillator $H=-\Delta+|x|^2$ as the (self-adjoint) comparison operator.
For a larger class of (quadratic) differential operators $A$, \cite{Alphonse-20} proved that the (negative of the) Laplacian $H=-\Delta$ is a suitable comparison operator.
However, the result of \cite{Alphonse-20} is strictly weaker for operators $A$ that are, at the same time,
comparable with the harmonic oscillator \emph{and} with the Laplacian:
Indeed, while spectral inequalities for the Laplacian require thick sensors sets, the results of \cite{MartinPS-20-arxiv,DickeSV-21}
show that thickness is not necessary for spectral inequalities for the harmonic oscillator.
In fact, our companion paper exhibited sensor sets that are even allowed to have finite measure, see \cite[Example~2.3]{DickeSV-21}
and also \eqref{eq:decay-twosided} above.

One of the goals of this paper is to propose a new class of comparison operators that interpolate, in some sense, between the negative
of the Laplacian and the harmonic oscillator.
In particular, our results imply that thickness of the sensor set can be relaxed if the comparison operator is not $-\Delta$.
The sensors sets we study in this situation were not covered before, except in the case of
the harmonic oscillator, when they agree with those studied in our companion paper \cite{DickeSV-21}.

In order to formulate our results, we need to introduce the notion of quadratic differential operators.
Let
\be\label{eq:quadratic-symbol}
	q
	\colon
	\RR^d\times\RR^d \to \CC
	,\quad
	q(x,\xi)
	=
	\sum_{\substack{|\alpha+\beta| = 2\\\alpha,\beta\in\NN_0^d}} c_{\alpha,\beta} x^\alpha\xi^\beta
	,\quad
	c_{\alpha,\beta}\in\CC
	,
\ee
be a complex quadratic form.
It is well known, see \cite{NicolaR-10}, that the distribution kernel
\[
	K(x,y)
	=
	(2\pi)^{-d/2} \cF^{-1}\bigl( \RR^d\ni\xi \mapsto q((x+y)/2,\xi)\bigr)(x-y)
\]
defines a continuous operator $q^\w\colon\cS(\RR^{d})\to\cS'(\RR^d)$ by
\[
	\langle q^\w u,v\rangle
	=
	\langle K,v\otimes u\rangle
	\quad\text{for}\quad
	u,v\in\cS(\RR^d)
	.
\]
Here $\cF\colon\cS'(\RR^d)\to\cS'(\RR^d)$ denotes the Fourier transform, $\langle \cdot,\cdot\rangle$ the pairing between $\cS'(\RR^d)$ and $\cS(\RR^d)$,
and $\cS'(\RR^{2d})$ and $\cS(\RR^{2d})$, respectively, and $\otimes$ the tensor product.
Moreover, the thus defined operator $q^\w$ extends to a continuous operator on $\cS'(\RR^d)$, see~\cite[Proposition~1.2.13]{NicolaR-10},
and we may therefore define the operator
\[
	A \colon L^2(\RR^d)\supset \cD(A) \to L^2(\RR^d)
	,
	\quad
	f\mapsto q^\w f
\]
on
\[
	\cD(A)
	=
	\{ f \in L^2(\RR^d) \colon q^\w f \in L^2(\RR^d) \}
	.
\]
We call $A$ the \emph{quadratic differential operator} associated to $q$ and $q$ its \emph{symbol}.

Note that the above construction does not require $q$ to be quadratic but works analogously for more general functions,
in particular \emph{any} polynomial.
However, if $q$ is a quadratic polynomial, the operator $A$ fits nicely into the general framework
of semigroups, which makes it well accessible for control theory:

\begin{proposition}[{see \cite{Hoermander-95}}]\label{prop:m-dissipative}
	Let $q$ be as in \eqref{eq:quadratic-symbol}.
	Then the operator $A$ is closed, densely defined, and agrees with the closure of the restriction of $q^\w$ to the space $\cS(\RR^d)$. If $\Re q \leq 0$,
	then $A$ is m-dissipative and generates a contraction semigroup.
\end{proposition}

Throughout the rest of this work, we will assume that $q$ is of the form \eqref{eq:quadratic-symbol} satisfying $\Re q \leq 0$.
We denote by $A$ the corresponding quadratic differential operator and by $(\cT(t))_{t\geq 0}$ the semigroup generated by $A$.

A particular example of a quadratic differential operator is the negative of a \emph{partial harmonic oscillator}.

\begin{definition}
	Let $\cI \subset \{1,\dots,d \}$, and let $A$ be the (self-adjoint) quadratic differential  operator with symbol
	\bes
		q(x,\xi) = q_\cI(x,\xi) = -|\xi|^2 - |x_\cI|^2,
		\quad
		|x_\cI|^2 = \sum_{j\in\cI}x_j^2.
	\ees
We call $H_\cI:= -A$ a partial harmonic oscillator.
\end{definition}

The latter can alternatively also be introduced via quadratic forms, see Appendix~\ref{sec:partharmOsc}, which leads to
the same operator since both agree on Schwartz functions; cf.~Corollary~\ref{cor:specH}.

Particular cases of partial harmonic oscillators are the negative of the Laplacian (with $\cI = \emptyset$) and the usual harmonic oscillator (with $\cI = \{1,\dots,d\}$).
These two are prominent prototypes for certain classes of quadratic differential operators.
In order to characterize these classes, we introduce the \emph{Hamilton map} associated to the quadratic form $q$
defined by
\[
	F
	=
	\frac{1}{2}
	\begin{pmatrix}
		(\partial_{\xi_j}\partial_{x_k} q(x,\xi))_{j,k=1}^d
		&
		(\partial_{\xi_j}\partial_{\xi_k} q(x,\xi))_{j,k=1}^d \\
		-(\partial_{x_j}\partial_{x_k} q(x,\xi))_{j,k=1}^d
		&
		-(\partial_{x_j}\partial_{\xi_k} q(x,\xi))_{j,k=1}^d
	\end{pmatrix}
	.
\]
Note that $F$ is a constant matrix since $q$ is a quadratic polynomial.
Associated to the Hamilton map is the so-called \emph{singular space} of the quadratic form $q$, or the operator $A$.
This was introduced in \cite{HitrikPS-09} as
\[
	S = S(A) = S(q)
	=
	\Biggl( \bigcap_{j = 0}^{2d-1} \ker\bigl[ \Re F (\Im F)^j\bigr] \Biggr)\cap \RR^{2d}
	,
\]
where $\Re F$ and $\Im F$ are taken entrywise.\footnote{Formally, $q$ and $-q$ generate the same singular space. For this reason (and notational simplicity) we set $S(-A)=S(A)$.}
We denote by $k_0 \in \{0,\dots 2d-1\}$ the smallest number such that
\be\label{eq:k_0}
	S
	=
	\Biggl( \bigcap_{j = 0}^{k_0} \ker\bigl[ \Re F (\Im F)^j\bigr] \Biggr)\cap \RR^{2d}
	.
\ee
For the purpose of this paper, we call $k_0$ the \emph{rotation exponent} of $q$ (resp.~$A$).
(It resembles somewhat the degree of non-holonomy in sub-Riemannian geometry.)

It turns out that it is reasonable to classify quadratic differential operators by the form of their singular space.
For instance, \cite{BeauchardPS-18} shows that
all quadratic differential operators $A$ with $S(A) = S(\Delta-|x|^2)$ satisfy a dissipation
estimate with respect to projections onto spectral subspaces of the harmonic oscillator.
Note that a simple calculation verifies $S(\Delta-|x|^2) = \{ 0 \}$.

\begin{proposition}[{\cite[Proposition~4.1]{BeauchardPS-18}}]\label{prop:dissipation-hermite}
	Let $S(A) = \{0\}$ and let $k_0$ be the rotation exponent of $A$.
	Then there are $c_0, t_0 > 0$ such that
	\be\label{eq:dissipation-hermite}
		\norm{(1-P_\lambda)\cT(t)g}_{L^2(\RR^d)}
		\leq
		c_0 \euler^{-c_0t^{2k_0+1}\lambda}\norm{g}_{L^2(\RR^d)}
	\ee
	for all $0 < t < t_0$, $\lambda\geq 1$, and $g\in L^2(\RR^d)$ where
	\[
		P_\lambda
		=
		\indic_{(-\infty,\lambda]} (-\Delta+|x|^2)
	\]
	is the projection onto the spectral subspace of the harmonic oscillator associated to the interval
	$(-\infty,\lambda]$.
\end{proposition}

A similar result is also available for the Laplacian:
In \cite[Remark~2.9]{AlphonseB-21} the authors state that the technique developed in \cite[Section~4.2]{Alphonse-20}
implies that all quadratic differential operators $A$ with $S(A) \subset S(\Delta)$ satisfy a dissipation estimate similar to \eqref{eq:dissipation-hermite},
but with $P_\lambda$ a projection onto a spectral subspace of the Laplacian;
note that $S(\Delta) = \RR^d\times\{0\}$.
This approach yields the following result which is, however, not formulated in the last mentioned references explicitly.

\begin{proposition}[{see \cite[Section~4.2]{Alphonse-20}, \cite[Remark~2.9]{AlphonseB-21}}]\label{prop:dissipation-general-S-Laplace}
	Assume $S(A) = U\times\{0\}$ for some subspace $U\subset\RR^d$ and let $k_0$ be the rotation exponent of $A$.
	Then there are $c_0, t_0 > 0$ such that
	\[
		\norm{(1-P_\lambda)\cT(t)g}_{L^2(\RR^d)}
		\leq
		c_0 \euler^{-c_0t^{2k_0+1}\lambda}\norm{g}_{L^2(\RR^d)}
	\]
	for all $0 < t < t_0$, $\lambda\geq 1$, and $g\in L^2(\RR^d)$ where
	\[
		P_\lambda
		=
		\indic_{(-\infty,\lambda]}(-\Delta)
	\]
	is the projection onto the spectral subspace of the negative of the Laplacian associated to the interval $(-\infty,\lambda]$.
\end{proposition}

As already mentioned above, the choice of the comparison operators determines the geometric
assumptions required for sensor sets.
The following result formulates a spectral inequality for spectral projectors of the harmonic oscillator
and thus complements the dissipation estimate in Proposition~\ref{prop:dissipation-hermite}.
Here we write $\Lambda_L(x) = x + (-L/2,L/2)^d$ for the cube of sidelength $L>0$ centered at $x\in\RR^d$.

\begin{proposition}[{\cite[Theorem~2.1]{DickeSV-21}}]\label{prop:ucp-hermite}
	There is a universal constant $K \geq 1$ such that for every
	$a \in [0,1)$, $L> 0$, $\gamma \in (0,1)$, $\lambda\in [1,\infty)$ and
	every measurable $\omega \subset \RR^d$ satisfying
	\bes
		\frac{\abs{\omega \cap \Lambda_L(m)}}{\abs{\Lambda_L(m)}}
		\geq
		\gamma^{1+\abs{m}^a}
		\quad\text{for all}\
		m \in (L\ZZ)^d
	\ees
	we have
	\bes
		\norm{f}_{L^2(\omega)}^2
		\geq
		3\Bigl( \frac{\gamma}{K^d} \Bigr)^{K d^{5/2+a} (1+L)^2\lambda^{(1+a)/2}}
		\norm{f}_{L^2(\RR^d)}^2
		\quad \text{for all}\quad f \in \Ran P_\lambda.
	\ees
	Here $P_\lambda$ is as in Proposition~\ref{prop:dissipation-hermite}.
\end{proposition}

On the other hand, in order to complement the dissipation estimate from Proposition~\ref{prop:dissipation-general-S-Laplace}
and obtain a spectral inequality for spectral projections of the Laplacian, we rely
on very precise uncertainty relations established in the seminal works \cite{Kovrijkine-01,Kovrijkine-thesis}.
While they were formulated by Kovrijkine in a Fourier analytic setting, it was observed in \cite{EgidiV-18,EgidiV-20}
that they translate to spectral inequalities for the Laplacian and have applications in control theory.
Analogous results for the Laplacian on finite cubes with periodic, Dirichlet or Neumann boundary conditions,
were obtained in \cite{EgidiV-18, EgidiV-20}, while \cite{EgidiS-21} established a more general spectral inequality
covering both bounded and unbounded domains.
We use here the specific formulation from \cite[Corollary~1.5]{EgidiS-21}:

\begin{proposition}[Kovrijkine's Inequality]\label{prop:ucp-laplace}
	Let $\omega \subset \RR^d$ be measurable satisfying
	\bes
		\frac{\abs{\omega \cap \Lambda_L(m)}}{\abs{\Lambda_L(m)}}
		\geq
		\gamma
		\quad\text{for all}\quad
		m \in (L\ZZ)^d
	\ees
	with some fixed $L > 0$ and $\gamma \in (0,1)$.
	
	Then, there is a universal constant $K \geq 1$ such that for every  $\lambda\in [1,\infty)$ and all
	$f \in \Ran P_\lambda$, where $P_\lambda$ is as in Proposition~\ref{prop:dissipation-general-S-Laplace}, we have
	\bes
		\norm{f}_{L^2(\omega)}^2
		\geq
		\Bigl( \frac{\gamma}{K^d} \Bigr)^{K d L \lambda^{1/2} + 2 d + 6}
		\norm{f}_{L^2(\RR^d)}^2		
		.
	\ees
\end{proposition}

Both dissipation estimates spelled out in Propositions~\ref{prop:dissipation-hermite} and \ref{prop:dissipation-general-S-Laplace}
cover the case $S(A)= \{0\}$. Hence, it is natural to compare the two complementing spectral inequalities in
Propositions~\ref{prop:ucp-hermite} and~\ref{prop:ucp-laplace} in this case;
clearly, the requirement on $\omega$ in Proposition~\ref{prop:ucp-hermite} is less restrictive.
In this sense, if $S(A)= \{0\}$, using the harmonic oscillator
as a comparison operator for $A$ allows for more general sensors sets than using the pure Laplacian.
This suggests that also for $S(\Delta)\supsetneq S(A) \supsetneq S(\Delta-|x|^2)$, that is,
\[
	\RR^d\times\{0\}
	\supsetneq
	S(A)
	\supsetneq
	\{0\}
	,
\]
there are better comparison operators than the Laplacian.
%
%
\section{Main results}\label{sec:main_results}

Our first main result on the way to establish observability is a dissipation estimate
that allows to treat quadratic differential operators $A$ with singular space $S = \RR^d_\cN\times\{0\}$.
Here we set
\[
	M_\cN^d
	=
	\{ x\in M^d\colon x_j = 0 \quad\text{for all}\quad j\notin\cN\}
\]
for $M\subset \RR$ and $\cN\subset\{1,\dots,d\}$.
Note that by a suitable rotation also more general singular spaces of the form $S = U \times \{ 0 \}$
for some subspace $U\subset\RR^d$ can be handled, see Subsection~\ref{ssec:crooked} below.

We consider here quadratic differential operators with $S(A) = S(H_\cI)$, where
\[
	H_{\cI}
	=
	-\Delta + |x_\cI|^2
	,
	\quad
	\cI\subset \{1,\dots,d\}
	,
\]
is the partial harmonic oscillator with singular space $S(H_\cI) = \RR^d_{\cI^\complement}\times\{0\}$.
The equality $S(A) = S(H_\cI)$ plays a crucial role in the smoothing estimates underlying the proof of the dissipation estimate.
We discuss this in detail in Section~\ref{sec:dissipation} together with the input we use from \cite{Alphonse-20, AlphonseB-21}.

In order to formulate the dissipation estimate we denote by
\be\label{eq:spectral-projection-partial-harmonic-oscillator}
  	P_\lambda
	=
	\indic_{(-\infty,\lambda]}(H_{\cI})
\ee
the projection onto the spectral subspace of $H_{\cI}$ associated to the interval $(-\infty,\lambda]$.

\begin{theorem}[Dissipation estimate]\label{thm:dissipation}
	Let $S(A) = S(H_\cI) = \RR^d_{\cI^\complement} \times \{ 0 \}$ for some set $\cI \subset \{1,\dots,d\}$.
	Then there are $c_0, t_0 > 0$ such that
	\[
		\norm{(1-P_\lambda)\cT(t)g}_{L^2(\RR^d)}
		\leq
		c_0 \euler^{-c_0t^{2k_0+1}\lambda}\norm{g}_{L^2(\RR^d)}
	\]
	for all $0 < t < t_0$, $\lambda\geq 1$, and $g\in L^2(\RR^d)$
	where $P_\lambda$ is the projection in \eqref{eq:spectral-projection-partial-harmonic-oscillator} and $k_0$ is the rotation exponent from \eqref{eq:k_0}.
\end{theorem}

Theorem~\ref{thm:dissipation}, or rather the more general Corollary~\ref{thm:general-dissipation}
below, covers and extends all previous dissipation estimates
obtained in \cite{BeauchardPS-18,Alphonse-20,MartinPS-20-arxiv}.

The second main result is a tailored spectral inequality
complementing the dissipation estimate in Theorem~\ref{thm:dissipation}.
Here, it turns out that the anisotropy of the potential of the partial harmonic oscillator $H_{\cI}$
translates into decay properties of functions in its spectral subspace $\Ran P_\lambda$:
The functions exhibit decay in those coordinate directions where the potential $V(x) = \sum_{j\in\cI}x_j^2$ grows.
This decay allows us to prove the spectral inequality without requiring thickness on the sensor set $\omega$.
In the particular case $\cI = \{1,\dots,d\}$ we recover Proposition~\ref{prop:ucp-hermite}.

\begin{theorem}[Spectral inequality]\label{thm:ucp}
	Let $\omega \subset \RR^d$ be measurable satisfying
	\be\label{eq:control-set-fixed-cubes}
		\frac{\abs{\omega \cap \Lambda_L(m)}}{\abs{\Lambda_L(m)}}
		\geq
		\gamma^{1+\abs{m_\cI}^a}
		\quad\text{for all}\
		m \in (L\ZZ)^d
	\ee
	with some fixed $a \in [0,1)$, $L > 0$, and $\gamma \in (0,1)$.

	Then, there is a universal constant $K \geq 1$ such that for every  $\lambda\in [1,\infty)$ and all
	$f \in \Ran P_\lambda$, with $P_\lambda$ as in \eqref{eq:spectral-projection-partial-harmonic-oscillator}, we have
	\be\label{eq:specIneq-fixed-cubes}
		\norm{f}_{L^2(\omega)}^2
		\geq
		3\Bigl( \frac{\gamma}{K^d} \Bigr)^{K d^{1+a} (1+L)^2\lambda^{(1+a)/2}}
		\norm{f}_{L^2(\RR^d)}^2
		.
	\ee
\end{theorem}

A result allowing more general sets $\omega$ is deferred to Corollary~\ref{cor:specIneq:genSets} below.

\begin{example}\label{ex:specIneq-fixed-cubes}
	Let $a,\gamma \in (0,1)$, $L=1$, and set
	\[
		\omega
		=
		\bigcup_{k\in\ZZ^d} \Lambda_{r_k}(k)\quad\text{with}\quad r_k = \gamma^{(1+\abs{k_\cI}^a)/d}
		.
	\]
	Then, $\omega$ satisfies
	\[
		\frac{\abs{\omega \cap \Lambda_1(k)}}{\abs{\Lambda_1(k)}}
		=
		\gamma^{1 + \abs{k_\cI}^a}
		\quad\text{ for all}\
		k \in \ZZ^d
		,
	\]
	so that the hypotheses of Theorem~\ref{thm:ucp} are satisfied, and we obtain
	\bes
		\norm{f}_{L^2(\omega)}^2
		\geq
		3\Bigl( \frac{\gamma}{K^d} \Bigr)^{4K d^{1+a} \lambda^{(1+a)/2}} \norm{f}_{L^2(\RR^d)}^2
	\ees
	for all $f \in \Ran P_{(-\infty,\lambda]}(H_\cI)$, $\lambda \geq 1$.
	Note that $\gamma_1=(1+a)/2< 1$, while on the other hand, the set $\omega$ is \emph{not} thick in $\RR^d$.
	In particular, if $\cI = \{1,\dots,d\}$, it even has finite measure since $\gamma \in (0,1)$.
	The latter holds also for $\omega$ as in \eqref{eq:decay-twosided}.
\end{example}

\begin{remark}\label{rem:fractional-operator}
	Analogously to \cite{NakicTTV-20} for the fractional Laplacian,
	we can also treat the fractional harmonic oscillator $H_\cI^\theta = (-\Delta+|x_\cI|^2)^\theta$ for certain $\theta>1/2$.
	More precisely, if $\omega$ satisfies \eqref{eq:control-set-fixed-cubes},
	then \eqref{eq:specIneq-fixed-cubes} implies
	\bes
		\norm{f}_{L^2(\omega)}^2
		\geq
		3\Bigl( \frac{\gamma}{K^d} \Bigr)^{K d^{1+a} (1+L)^2\lambda^{\frac{1+a}{2\theta}}}
		\norm{f}_{L^2(\RR^d)}^2
	\ees
	for $f \in  \Ran P_{(-\infty,\lambda]}(H_\cI^\theta)= \Ran P_{(-\infty,\lambda^{1/\theta}]}(H_\cI)$, relying on
	the transformation formula for spectral measures.
	This requires $\theta>(1+a)/2$ in order to guarantee $\gamma_1=\frac{1+a}{2\theta} <1$ in \eqref{eq:spectral}.
\end{remark}

The combination of the dissipation estimate in Theorem~\ref{thm:dissipation} and the spectral inequality in Theorem~\ref{thm:ucp}
implies by the Lebeau-Robbiano method the following observability result.
In view of Example~\ref{ex:specIneq-fixed-cubes} it sharpens \cite[Theorem~1.12]{Alphonse-20}.

\begin{theorem}[Observability]\label{thm:observability}
	Let $A$ be a quadratic differential operator on $\RR^d$ with singular space $S(A) = S(H_\cI) = \RR^d_{\cI^\complement} \times \{ 0 \}$ for some set
	$\cI \subset \{1,\dots,d\}$, and let $\omega$ be as in Theorem~\ref{thm:ucp}.
	Then the abstract Cauchy problem \eqref{eq:controlled-abstract-Cauchy-problem} with $\cOp = \indic_\omega$ is final-state observable.
\end{theorem}

\begin{remark}
	It is also possible to treat operators $A$ with $S(A) = U \times \{0\}$, where $U$ is some subspace of $\RR^d$.
	This allows, e.g., to consider for $d=2$ the symbol $q(x,\xi) = -|\xi|^2-(x_1+x_2)^2$ on $L^2(\RR^2)$, the singular
	space of which is of the above form with $U = \{r\cdot (1,-1)^\top\colon r\in\RR\}$,
	cf.~Subsection~\ref{ssec:crooked} below.
\end{remark}

We close this section by discussing a null-controllability result that follows from the facts discussed so far.

It is well known that the adjoint $A^*$ of $A$ is again a quadratic differential operator with symbol $\overline{q}$, see~\cite[Proposition~1.2.10]{NicolaR-10}.
Since $q$ and $\overline{q}$ have the same singular space and the same rotation exponent,
this establishes that also the abstract Cauchy problem
\bes
	x'(t)
	=
	A^*x(t)
	,
	\quad
	y(t)
	=
	\indic_\omega x(t)
	,
	\quad
	x(0)
	=
	x_0
	,
\ees
corresponding to $A^*$ is observable from $\omega$.
By the well-known Hilbert uniqueness method this implies null-controllability in time $T > 0$ of
\be\label{eq:abstract-Cauchy-problem-control}
	w'(t)
	=
	Aw(t) + \indic_\omega u(t)
	,
	\quad
	w(0)
	=
	w_0
	,
\ee
that is, for all $w_0 \in L^2(\RR^d)$ there exists $u \in L^2((0,T);L^2(\RR^d))$ such that the mild solution
\[
	w(t)
	=
	\cT(t)w_0 + \int_0^t \cT(t-s)\indic_\omega u(s)\Diff{s}
\]
to \eqref{eq:abstract-Cauchy-problem-control} satisfies $w(T) = 0$,
see also, e.g., \cite{Zuazua-06,Coron-07,TucsnakW-09,EgidiNSTTV-20,NakicTTV-20} and the references cited therein.
In this case, the so-called \emph{control cost}
\[
	C_T
	=
	\sup_{\norm{w_0} = 1} \inf\bigl\{ \norm{u}_{L^2((0,T);L^2(\RR^d))} \colon w(T) = 0\bigr\}
\]
satisfies
\[
	C_T
	\leq
	C_{\mathrm{obs}}
	,
\]
and the latter is finite by Theorem~\ref{thm:observability}.
Thereby, we arrive at the following

\begin{corollary}[Null-controllability]\label{cor:control}
	Let $A$ be a quadratic differential operator on $\RR^d$ with singular space $S(A) = S(H_\cI) = \RR^d_{\cI^\complement} \times \{ 0 \}$ for some set
	$\cI \subset \{1,\dots,d\}$, and let $\omega$ be as in Theorem~\ref{thm:ucp}.
	Then the abstract Cauchy problem \eqref{eq:controlled-abstract-Cauchy-problem} with $\cOp = \indic_\omega$ is null-controllable.
\end{corollary}
%
%
\section{Dissipation estimate}\label{sec:dissipation}

Recall that $(\cT(t))_{t\geq 0}$ is a strongly continuous contraction semigroup and that
its generator $A$ is a quadratic differential operator corresponding to a quadratic symbol $q$ with $\Re q \leq 0$.

\subsection{Smoothing effects}

The proof of Theorem~\ref{thm:dissipation} is based on so-called \emph{smoothing effects} of the semigroup $(\cT(t))_{t\geq 0}$.
These describe the fact that for appropriate quadratic symbols $q$ the function $\cT(t)g\in L^2(\RR^d)$, $t > 0$, has a certain regularity
for every $g\in L^2(\RR^d)$.
Several recent works, see, e.g., ~\cite{HitrikPSV-18, Alphonse-20, AlphonseB-21}, show that the smoothing effects of the semigroup
are closely related to the structure of the singular space $S$.
One of the first results, \cite[Proposition~3.1.1]{HitrikPS-09}, shows that for $S(A) = \{ 0 \}$
we have $\cT(t)g\in\cS(\RR^d)$ for all $g \in L^2(\RR^d)$ and $t > 0$.

For comparison, we first state the result that was the main ingredient in the proof of
the dissipation estimate in \cite{BeauchardPS-18}, formulated in Proposition~\ref{prop:dissipation-hermite} above.

\begin{proposition}[{\cite[Proposition~4.1]{HitrikPSV-18}}] \label{prop:smoothing-harmonic-oscillator}
	Let $S(A) = \{0\}$ and let $k_0$ be the rotation exponent from \eqref{eq:k_0}.
	Then there are $c_0, c_0', t_0 > 0$ such that
	\bes
		\norm{\euler^{c_0t^{2k_0+1}(-\Delta+|x|^2)}\cT(t)g}_{L^2(\RR^d)}
		\leq
		c_0' \norm{g}_{L^2(\RR^d)}
		\quad\text{for all}\quad
		0 < t < t_0
		.
	\ees
\end{proposition}

The last inequality implies (cf.~\cite[Inequality~(4.19)]{HitrikPSV-18}) that for some $C > 0$ we have
\be\label{eq:Corollary-HitrikPSV-18}
	\norm{x^\alpha \partial_x^\beta \cT(t)g}_{L^2(\RR^d)}
	\leq
	\frac{C^{1+|\alpha|+|\beta|}(\alpha!)^{1/2}(\beta!)^{1/2}}{t^{(k_0+1/2)(|\alpha|+|\beta|+2d)}}\norm{g}_{L^2(\RR^d)}
\ee
for all $\alpha, \beta \in \NN_0^d$ and $0 < t < t_0$.
This establishes that the semigroup is smoothing in the so-called Gelfand-Shilov space $S^{1/2}_{1/2}(\RR^d)$.
For the definition of the general Gelfand-Shilov spaces $S^\nu_\mu(\RR^d)$, $\mu,\nu > 0$, see
\cite[Chapter~6]{NicolaR-10}.

An alternative proof of Proposition~\ref{prop:smoothing-harmonic-oscillator} using \eqref{eq:Corollary-HitrikPSV-18} has been suggested in \cite{MartinPS-20-arxiv}.
In a similar way, as observed in \cite[Remark~2.9]{AlphonseB-21}, the technique of \cite[Section~4.2]{Alphonse-20} can be adapted to prove
\bes
		\norm{\euler^{c_0t^{2k_0+1}(-\Delta)}\cT(t)g}_{L^2(\RR^d)}
		\leq
		c_0' \norm{g}_{L^2(\RR^d)}
		\quad\text{for all}\quad
		0 < t < t_0
\ees
using a variant of \eqref{eq:Corollary-HitrikPSV-18} with $\alpha = 0$.

We follow the same path and establish a variant of Proposition~\ref{prop:smoothing-harmonic-oscillator} for the partial harmonic oscillator $H_\cI$.
To this end, we need the following corollary to \cite[Theorem~2.6]{AlphonseB-21}.
In the formulation of this result, the orthogonality is taken with respect to the usual Euclidean inner product on $\RR^{2d}$.
Here we write $\NN_{0,\cI}^d = (\NN_0)^d_{\cI}$ for simplicity.

\begin{theorem}\label{thm:general-smoothing}
	Let $S(A)^\perp = \RR^d_\cI \times \RR^d_{\cJ}$ for some sets $\cI,\cJ \subset \{1,\dots,d\}$
	and let $k_0$ be the rotation exponent from \eqref{eq:k_0}.
	Then there are constants $C > 0$ and $t_0 \in (0,1)$ such that for all $\alpha\in\NN_{0,\cI}^d$, $\beta\in\NN_{0,\cJ}^d$, and
	$0 < t < t_0$ we have
	\be\label{eq:general-smoothing}
		\norm{x^\alpha \partial_x^\beta \cT(t)g}_{L^2(\RR^d)}
		\leq
		\frac{C^{|\alpha| + |\beta|}}{t^{(|\alpha| + |\beta|)(k_0+1/2)}}(\alpha!)^{1/2}(\beta!)^{1/2}\norm{g}_{L^2(\RR^d)}
		.
	\ee
\end{theorem}

\begin{proof}
	We set $Y_k = (e_k,0)$, $k\in\cI$, and $Y'_j = (0,e_j)$, $j\in\cJ$, where $e_l \in \RR^d$ is the $l$-th unit vector in $\RR^d$.
	Denote by  $D_{Y_k}$ and $D_{Y'_j}$ the Weyl quantization of the symbols $q(x,\xi) = e_k\cdot x $ and $q(x,\xi) = e_j\cdot \xi $, respectively,
	i.e.~$D_{Y_k} = x_k$ and $D_{Y'_j} = -\ii\partial_j$,

	Under the imposed assumptions, \cite[Theorem~2.6]{AlphonseB-21} implies that there are constants $C_0 > 0$ and $t_0 \in (0,1)$
	such that for all $m \in \NN$ and $0 < t < t_0$ we have
	\be\label{eq:smoothing-properties-differential-expressions}
		\norm{D_{Y^1}\dots D_{Y^m}\cT(t)g}_{L^2(\RR^d)}
		\leq
		\frac{C_0^m}{t^{mk_0+m/2}}(m!)^{1/2}\norm{g}_{L^2(\RR^d)}
		.
	\ee
	Here each of the $Y^1,\dots,Y^m$  can be any of the
	vectors $\{(e_k,0)_{k\in\cI},(0,e_j)_{j\in\cJ}\}$ forming a basis of $S(A)^\perp = \RR^d_\cI \times \RR^d_{\cJ}$.

	Let $\alpha\in\NN_{0,\cI}^d$ and $\beta\in\NN_{0,\cJ}^d$.
	For each $k\in\cI$, we take $\alpha_k$-times the vector $Y_k$, and, similarly, $\beta_j$-times the vector $Y'_j$ for each $j\in\cJ$.
	In total, these are $m = |\alpha| + |\beta|$ many vectors.
	Hence \eqref{eq:smoothing-properties-differential-expressions} implies
	\eqs{
		\norm{x^\alpha \partial_x^\beta \cT(t)g}_{L^2(\RR^d)}
		&\leq
		\frac{C_0^{|\alpha| + |\beta|}}{t^{(|\alpha| + |\beta|)(k_0+1/2)}}((|\alpha| + |\beta|)!)^{1/2}\norm{g}_{L^2(\RR^d)}
		\\
		&\leq
		\frac{C^{|\alpha| + |\beta|}}{t^{(|\alpha| + |\beta|)(k_0+1/2)}}(\alpha!)^{1/2}(\beta!)^{1/2}\norm{g}_{L^2(\RR^d)}
		,
	}
	where $C = 2\sqrt{d}C_0$.
\end{proof}%

Note that inequality~\eqref{eq:general-smoothing} shows that the singular space
encodes the directions in which one expects a certain decay of the function $\cT(t)g$ for fixed $t > 0$ and $g \in L^2(\RR^d)$.

\subsection{Proof of the dissipation estimate}

We now show that Theorem~\ref{thm:general-smoothing} implies a version of Proposition~\ref{prop:smoothing-harmonic-oscillator} for the partial harmonic
oscillator. To this end, let $\cI,\cJ \subset \{ 1,\dots,d \}$, and consider the
operator $H_{\cI,\cJ}$ corresponding to the
differential expression $-\Delta_\cJ+|x_\cI|^2=-\sum_{j\in\cJ}\partial_j^2+\sum_{i\in\cI}x_i^2$ defined via quadratic forms, see Appendix~\ref{sec:partharmOsc}.
In view of inequality \eqref{eq:general-smoothing} we single out the following class of \emph{partially Schwartz functions}
\[
	\cG_{\cI,\cJ}
	:=
	\{ f \in L^2(\RR^d) \colon x^\alpha \partial^\beta f \in L^2(\RR^d)\ \forall \alpha \in \NN_{0,\cI}^d,\, \beta \in \NN_{0,\cJ}^d \}
	,
\]
and denote $l := \abs{\cI \cap \cJ}\leq d$. In view of Theorem~\ref{thm:general-smoothing}
the assumptions of the following lemma are natural.

\begin{lemma}\label{lem:technical-lemma-dissipation}
	Let $D_1,D_2 > 0$ be constants, and suppose that $f \in \cG_{\cI,\cJ}$ satisfies
	\be\label{eq:assumption-D1D2}
		\norm{x^\alpha \partial^\beta f}_{L^2(\RR^d)}
		\leq
		D_1 D_2^{\abs{\alpha}+\abs{\beta}} (\alpha!)^{1/2} (\beta!)^{1/2}
		\text{ for all }
		\alpha \in \NN_{0,\cI}^d, \beta \in \NN_{0,\cJ}^d
		.
	\ee
	Then, for $s \leq 1/(40\euler\cdot 2^{d}d D_2^2)$ we have $f \in \cD(\euler^{sH_{\cI,\cJ}})$ and
	\[
		\norm{\euler^{sH_{\cI,\cJ}}f}_{L^2(\RR^d)}
		\leq
		2\Bigl(\frac{2}{3}\Bigr)^l D_1
		\leq
		2 D_1		.
	\]
\end{lemma}

\begin{proof}
	On $\cG_{\cI,\cJ}$ define the differential expressions $T_j$, $j \in \{ 1,2,3 \}$, with
	\[
		T_1g = (-\Delta_{\cI \cap \cJ} + \abs{x_{\cI \cap \cJ}}^2)g,\quad
		T_2g = -\Delta_{\cJ \setminus \cI}g,\quad
		T_3g = \abs{x_{\cI \setminus \cJ}}^2g
	\]
	for $g \in \cG_{\cI,\cJ}$. By Lemma~\ref{lem:reprHamiltonian}, we have
	\[
		(H_{\cI,\cJ}+l)f
		=
		(T_1 + l)f + T_2f + T_3f
		.
	\]
	Since the $T_j$ leave $\cG_{\cI,\cJ}$ invariant and commute pairwise, this gives for $n \in \NN_0$
	\bes
		(H_{\cI,\cJ}+l)^nf
		=
		\sum_{\substack{\abs{\nu}=n\\ \nu \in \NN_0^3}} \binom{n}{\nu}  (T_1+l)^{\nu_1} T_2^{\nu_2} T_3^{\nu_3}f
		.
	\ees
	We have
	\[
		T_2^{\nu_2}T_3^{\nu_3}f
		=
		(-1)^{\nu_2} \sum_{\substack{|\beta| = \nu_2 \\ \beta \in \NN_{0,\cJ\setminus\cI}^d}}
		\sum_{\substack{|\alpha| = \nu_3\\\alpha\in\NN_{0,\cI\setminus\cJ}^d}}
		\binom{\nu_2}{\beta}\binom{\nu_3}{\alpha} x^{2\alpha} \partial_x^{2\beta}f
		\in
		\cG_{\cI,\cJ}(\RR^d)
		.
	\]
	Moreover, we recall from \cite[Eq.~(4.9) and (4.11)]{MartinPS-20-arxiv} that
	\be\label{eq:identity-harmonic-oscillator}
		(T_1 + l)^{\nu_1} g
		=
		\sum_{\substack{|\gamma+\delta|\leq 2 \nu_1\\\gamma,\delta\in\NN_{0,\cI\cap\cJ}^d}} c_{\gamma,\delta}^{(\nu_1)} x^\gamma\partial^\delta_xg
		,
		\quad
		g\in\cG_{\cI,\cJ}
		,
	\ee
	where the coefficients satisfy the bound
	\be\label{eq:coefficients-identity-harmonic-oscillator}
		|c_{\gamma,\delta}^{(\nu_1)}|
		\leq
		3^{2\nu_1-l}l^{\nu_1}(2\nu_1)^{(2\nu_1-|\gamma+\delta|)/2}
		,
		\quad
		|\gamma+\delta|
		\leq
		2 \nu_1
		.
	\ee
	Hence, inserting $g = T_2^{\nu_2}T_3^{\nu_3}f$ in formula \eqref{eq:identity-harmonic-oscillator} and using
	the triangle inequality for operator norms we are left with estimating
	\eqs{
		\norm{(T_1+&l)^{\nu_1}T_2^{\nu_2}T_3^{\nu_3}f}\\
		&\leq
		\sum_{\substack{|\gamma+\delta|\leq 2 \nu_1\\\gamma,\delta\in\NN_{0,\cI\cap\cJ}^d}}
		\sum_{\substack{|\beta| = \nu_2 \\ \beta \in \NN_{0,\cJ\setminus\cI}^d}}
		\sum_{\substack{|\alpha| = \nu_3\\\alpha\in\NN_{0,\cI\setminus\cJ}^d}}
		|c_{\gamma,\delta}^{(\nu_1)}|
		\binom{\nu_2}{\beta}\binom{\nu_3}{\alpha}
		\norm{x^{\gamma+2\alpha} \partial_x^{\delta+2\beta}f}
		.
	}
	Note that we can apply the hypothesis \eqref{eq:assumption-D1D2}
	for each summand separately with $\gamma+2\alpha\in\NN_{0,\cI}^d$ and $\delta+2\beta\in\NN_{0,\cJ}^d$.
	Hence, using also \eqref{eq:coefficients-identity-harmonic-oscillator}, we get
	\begin{multline*}
		|c_{\gamma,\delta}^{(\nu_1)}|\cdot\norm{x^{\gamma+2\alpha} \partial_x^{\delta+2\beta}f}
		\\
		\leq
		3^{2\nu_1-l}l^{\nu_1}(2\nu_1)^{(2\nu_1-|\gamma+\delta|)/2} D_1 D_2^{|\gamma+2\alpha|+|\delta+2\beta|} ((\gamma+2\alpha)!)^{1/2}((\delta+2\beta)!)^{1/2}
.
	\end{multline*}
	We further estimate this term using the simple inequality $\zeta! \leq |\zeta|^{|\zeta|}$ for multiindices $\zeta$,
	$|\gamma+2\alpha|+|2\beta+\delta|\leq 2n$ and $2\nu_1-|\gamma+\delta|+|\gamma+2\alpha|+|\delta+2\beta| = 2n$.
	Combining this with $(2n)^n\leq (2\euler)^nn!$ yields
	\eqs{
		3^{2\nu_1-l}&l^{\nu_1}(2\nu_1)^{(2\nu_1-|\gamma+\delta|)/2} D_1 D_2^{|\gamma+2\alpha|+|\delta+2\beta|} ((\gamma+2\alpha)!)^{1/2}((\delta+2\beta)!)^{1/2}
		\\
		&\leq
		3^{2\nu_1-l}d^{\nu_1}D_1 (2\euler D_2^2)^nn!
		.
	}
	Noting also
	\[
		\sum_{\substack{|\gamma+\delta|\leq 2\nu_1\\\gamma,\delta\in\NN_{0,\cI\cap\cJ}^d}}
		\sum_{\substack{|\beta| = \nu_2 \\ \beta \in \NN_{0,\cJ\setminus\cI}^d}}
		\sum_{\substack{|\alpha| = \nu_3\\\alpha\in\NN_{0,\cI\setminus\cJ}^d}}
		\binom{\nu_2}{\beta}\binom{\nu_3}{\alpha}
		\leq
		(2\nu_1 + 1)^l d^{\nu_2+\nu_3}
		\leq
		2^l2^{\nu_1 d}d^{\nu_2+\nu_3}
		,
	\]
	we finally derive
	\[
		\norm{ (T_1+l)^{\nu_1} T_2^{\nu_2}T_3^{\nu_3}f}
		\leq
		\Bigl(\frac{2}{3}\Bigr)^l D_1(9\euler\cdot 2^{d})^{\nu_1}(2\euler\cdot d D_2^2)^nn!
		.
	\]

	By the multinomial formula, we have thus shown
	\eqs{
		\norm{(H_{\cI,\cJ}+l)^nf}
		&\leq
		\sum_{\substack{\abs{\nu}=n\\ \nu \in \NN_0^3}}\binom{n}{\nu}\norm{ (T_1+l)^{\nu_1} T_2^{\nu_2}T_3^{\nu_3}f}
		\\
		&\leq
		\Bigl(\frac{2}{3}\Bigr)^l D_1(2\euler \cdot dD_2^2)^n(9\cdot 2^d+2)^n n!\\
		&\leq
		\Bigl(\frac{2}{3}\Bigr)^l D_1(20\euler\cdot 2^dD_2^2)^n n!
		.
	}
	Now, choose $1/s = 40\euler\cdot 2^{d}d D_2^2$.
	Then $f\in\cD(\euler^{s(H_{\cI,\cJ}+l)})$ and
	\[
		\norm{\euler^{s(H_{\cI,\cJ}+l)}f}
		\leq
		\sum_{n = 0}^\infty \frac{s^n}{n!}\norm{(H_{\cI,\cJ}+l)^n f}
		\\
		\leq
		2\cdot \Bigl(\frac{2}{3}\Bigr)^lD_1
		.
	\]
	It remains to observe that $f\in\cD(\euler^{s(H_{\cI,\cJ})})$ with
	$\norm{\euler^{sH_{\cI,\cJ}}f}\leq\norm{\euler^{s(H_{\cI,\cJ}+l)}f}$ by the spectral theorem.
\end{proof}%

The above lemma is the central tool in the proof of the next theorem, which is a generalization and sharpening of Proposition~\ref{prop:smoothing-harmonic-oscillator}.

\begin{theorem}\label{thm:L2-bound}
	Let $S(A)^\perp = \RR^d_{\cI}\times\RR^d_{\cJ}$ for some sets $\cI,\cJ \subset \{1,\dots,d\}$ and let
	$k_0$ be the rotation exponent from \eqref{eq:k_0}.
	Then we have for all $g \in L^2(\RR^d)$ that $\cT(t) g \in \cD(\euler^{ct^{2k_0+1}H_{\cI,\cJ}})$ and
	\bes
		\norm{\euler^{ct^{2k_0+1}H_{\cI,\cJ}}\cT(t)g}_{L^2(\RR^d)}
		\leq
		2\norm{g}_{L^2(\RR^d)}  \quad \text{ for all $0 < t < t_0$}
		.
	\ees
	Here $c = 1/(40\euler\cdot 2^{d}d C^2)$, and $C$ and $t_0 \in (0,1)$ are as in Theorem~\ref{thm:general-smoothing}.
\end{theorem}

\begin{proof}
	We observe that inequality \eqref{eq:general-smoothing} in Theorem~\ref{thm:general-smoothing}  shows
	that for every $0 < t < t_0$ the function $f = \cT(t)g$ satisfies the hypotheses of Lemma~\ref{lem:technical-lemma-dissipation} with
	$D_1 = \norm{g}_{L^2(\RR^d)}$ and $D_2 = C t^{-(k_0+1/2)}$.
	Lemma~\ref{lem:technical-lemma-dissipation} therefore gives
	\bes
		\norm{\euler^{sH_{\cI,\cJ}}f}_{L^2(\RR^d)}
		\leq
		2\norm{g}_{L^2(\RR^d)}
	\ees
	for $s \leq 1/(40\euler\cdot 2^{d}d D_2^2)$.
	This shows
	\[
		\norm{\euler^{ct^{2k_0+1}H_{\cI,\cJ}}\cT(t)g}_{L^2(\RR^d)}
		\leq
		2\norm{g}_{L^2(\RR^d)}
		,
	\]
	where $c = 1/(40\euler\cdot 2^{d}d C^2)$.
\end{proof}%

We have now assembled all tools needed to prove
a generalized version of our Theorem~\ref{thm:dissipation},
i.e.~a dissipation estimate for small times for the projections
\bes
	P_\lambda = P_{(-\infty,\lambda]}(H_{\cI,\cJ}).
\ees

\begin{theorem}\label{thm:general-dissipation}
	Let $S(A)^\perp = \RR^d_{\cI} \times \RR^d_{\cJ}$ for some sets $\cI,\cJ \subset \{1,\dots,d\}$
	and let $k_0$ be the rotation exponent from \eqref{eq:k_0}.
	Then, with constants $C > 0$ and $t_0 \in (0,1)$ as in  Theorem~\ref{thm:general-smoothing}, we have
	\[
		\norm{(1-P_\lambda)\cT(t)g}_{L^2(\RR^d)}
		\leq
		2\euler^{-ct^{2k_0+1}\lambda}\norm{g}_{L^2(\RR^d)}
		,
		\quad c = (40\euler \cdot 2^d  d C^2)^{-1}
		,
	\]
	for all $g \in L^2(\RR^d)$, $0 < t < t_0$, and $\lambda \geq 0$.
\end{theorem}

In the particular case $\cJ = \{ 1,\dots,d \}$ this agrees with Theorem~\ref{thm:dissipation}.
The proof follows the strategy of \cite[Proposition~4.1]{BeauchardPS-18}.

\begin{proof}
	By Theorem~\ref{thm:L2-bound} we have $\cT(t)g\in\cD(\euler^{ct^{2k_0+1}H_{\cI,\cJ}})$ and
	\be\label{eq:L2-bound-weighted-all-t}
		\norm{\euler^{ct^{2k_0+1}H_{\cI,\cJ}}\cT(t)g}_{L^2(\RR^d)}
		\leq
		2\norm{g}_{L^2(\RR^d)} \quad \text{ for $0 < t < t_0$.}
	\ee
	For those $t$, we therefore have
	\[
		\cT(t)g
		=
		\euler^{-ct^{2k_0+1}H_{\cI,\cJ}}\euler^{ct^{2k_0+1}H_{\cI,\cJ}}\cT(t)g
		.
	\]
	Moreover, the projections $P_\lambda$ and the operator $\euler^{-ct^{2k_0+1}H_{\cI,\cJ}}$ commute, so that the
	previous identity and the spectral theorem imply
	\eqs{
		\norm{(1-P_\lambda)\cT(t)g}
		&=
		\norm{\left[\euler^{-ct^{2k_0+1}H_{\cI,\cJ}} (1-P_\lambda)\right]\euler^{ct^{2k_0+1}H_{\cI,\cJ}}\cT(t)g}\\
		&\leq
	\norm{\left[\euler^{-ct^{2k_0+1}H_{\cI,\cJ}} (1-P_\lambda)\right]} \cdot \norm{\euler^{ct^{2k_0+1}H_{\cI,\cJ}}\cT(t)g}\\
		&\leq
	   2\euler^{-ct^{2k_0+1}\lambda} \cdot \norm{g}_{L^2(\RR^d)} \quad \text{ for $0 < t < t_0$,}
	}
	where we used inequality \eqref{eq:L2-bound-weighted-all-t} in the last line.
\end{proof}%

\subsection{Crooked singular spaces}\label{ssec:crooked}

The above considerations can also be used to treat
more general singular spaces of the form
$S(q)^\perp = V \times W$, where $V,W \subset \RR^d$ are vector spaces of
dimensions $d_1 = \dim V$ and $d_2 = \dim W$.
Indeed, in this case, there is an orthogonal transformation $\cR \colon \RR^d \to \RR^d$ such that
\be\label{eq:rotation-dissipation}
	\cR V
	=
	\RR^d_\cI
	\quad\text{and}\quad
	\cR W
	=
	\RR^d_\cJ
\ee
with
\[
	\cI
	=
	\{1,\dots,d_1\}
	\quad\text{and}\quad
	\cJ
	=
	\{d_1-l+1,\dots,d_1+d_2-l\}
	,
\]
where $l = \dim (V \cap W)$.
Then the singular space of the form $\tilde{q}$ given by $\tilde{q}(x,\xi) = q(\cR^{-1} x,\cR^{-1} \xi)$
for all $x,\xi\in\RR^d$ is characterized by
$S(\tilde{q})^\perp = \RR^d_{\cI} \times \RR^d_{\cJ}$,
and the accretive operators $A$ and $\tilde{A}$ associated with $q$ and $\tilde{q}$,
respectively, by the Weyl quantizations satisfy $A = \cU_\cR \tilde{A} \cU_\cR^{-1}$ where $\cU_\cR f = f\circ\cR$.
Using this construction we derive the following general result.

\begin{corollary}\label{cor:rotated-dissipation}
	Let $S(A)^\perp = V\times W$, and let $\cR$ be
	as in \eqref{eq:rotation-dissipation}.
	Then, with
	$P_\lambda = \cU_\cR P_{(-\infty,\lambda]}(H_{\cI,\cJ}) \cU_\cR^{-1}$,
	we have
	\[
		\norm{(1-P_\lambda)\cT(t)g}_{L^2(\RR^d)}
		\leq
		2\euler^{-ct^{2k_0+1}\lambda}\norm{g}_{L^2(\RR^d)}
	\]
	for all $g \in L^2(\RR^d)$, $0 < t < t_0$, and $\lambda \geq 0$.
\end{corollary}

\begin{remark}
	Note that there are also quadratic forms whose singular space does not satisfy $S(q)^\perp = V \times W$.
	Consider, e.g., the form $q(x,\xi) = -(x+\xi)^2$ on $\RR^2$ with singular space $S(q) = \{r\cdot (1,-1)^\top\colon r\in\RR\}$.
	Such forms are not covered by Corollary~\ref{cor:rotated-dissipation}.
\end{remark}
%
%
\section{Spectral inequality}\label{sec:spectral-inequality}

In this section, we prove a spectral inequality for the operator $H_\cI = H_{\cI,\cJ}$ with $\cJ = \{ 1,\dots,d \}$ and
$\cI \neq \emptyset$. To this end, we generalize the arguments in \cite{DickeSV-21}, where the special case $\cI = \{ 1,\dots,d \}$ was treated.

Without loss of generality, we may reorder the coordinates of $\RR^d$
such that we have $\cI = \{ 1,\dots,d_1 \}$ with some $1 \leq d_1 \leq d$.
Set $d_2 := d - d_1$.
We introduce the operators $H_1$ and $H_2$ corresponding to the differential
expressions
\[
	-\Delta + \abs{x}^2 \quad \text{in}\quad L^2(\RR^{d_1}),\quad
	-\Delta \quad \text{in}\quad L^2(\RR^{d_2})
	,
\]
respectively, via their quadratic forms; cf.\ Appendix~\ref{sec:partharmOsc}. In other words, $H_1$ is the harmonic
oscillator in $L^2(\RR^{d_1})$ and $-H_2$ is the pure Laplacian on $L^2(\RR^{d_2})$.

If $d_2=0$, we just have $H_\cI=H_1$. On the other hand, for $d_2 \neq 0$ we have by Lemma~\ref{lem:reprH},
Corollary~\ref{cor:specH} and Remark~\ref{rem:truePartHarm} that
\[
	H_\cI
	=
	H_1 \otimes I_2 + I_1 \otimes H_2
	\quad\text{and}\quad
	\sigma(H)
	=
	\sigma(H_1) + \sigma(H_2)
	\subset
	[d_1,\infty)
	,
\]
where for the latter we used that $\sigma(H_1) \subset [d_1,\infty)$ and $\sigma(H_2) = [0,\infty)$;
recall that $H_1$ has pure point spectrum with eigenvalues $2k+d_1$, $k \in \NN_0$.
Moreover, by Corollary~\ref{cor:specfamH}
every $f \in \Ran P_{(-\infty,\lambda]}(H_\cI)$ can be extended to an analytic function on $\CC^d$. We denote this extension
again by $f$.

Throughout this section, let $(Q_k)_{k \in \cK}$ be any finite or countably infinite family of measurable
subsets $Q_k \subset \RR^d$ and $\kappa \geq 1$ such that
\be\label{eq:covering-general-assumption}
	\Bigl|\RR^d \setminus \bigcup_{k \in \cK} Q_k\Bigr|=0\qquad \text{and}\qquad
	\sum_{k \in \cK} \indic_{Q_k} (x)
	\leq
	\kappa
	\quad \text{for all}\ x \in\RR^d.
\ee
We say that $(Q_k)_{k \in \cK}$ is an \emph{essential covering} of $\RR^d$ of \emph{multiplicity at most $\kappa$}.

As a starting point, we derive exponential decay of elements of the spectral subspace $\Ran P_{(-\infty,\lambda]}(H_\cI)$,
$\lambda \geq 1$, in the growth directions of the potential, that is, the coordinates $\cI$. To this end, we recall that the elements of
$\Ran P_{(-\infty,\lambda]}(H_1)$ are finite linear combinations of the well-known Hermite functions, which exhibit an
exponential decay in terms of a weighted $L^2$-estimate, see for instance \cite[Proposition~3.3]{BeauchardJPS-21}.
Using the tensor representation of $H_\cI$, we now obtain the
following result.

\begin{lemma}\label{lem:decay}
	For all $f \in \Ran P_{(-\infty,\lambda]}(H_\cI)$, $\lambda \geq 1$, we have
	\bes
		\norm{\euler^{|x_\cI|^2/64d_1}f}_{L^2(\RR^d)}^2
		\leq
		2^{2(d_1+1)+\lambda}\norm{f}_{L^2(\RR^d)}^2
		.
	\ees
\end{lemma}

\begin{proof}
	Let $f \in \Ran P_{(-\infty,\lambda]}(H_\cI)$ with $\lambda \geq 1$. By Corollary~\ref{cor:specfamH}, we have
	$f(\cdot,y) \in \Ran_{(-\infty,\lambda]}(H_1)$ for all $y \in \RR^{d_2}$, and by \cite[Proposition~3.3]{BeauchardJPS-21}
	\[
		\norm{\euler^{|\cdot|^2/64d_1}f(\cdot,y)}_{L^2(\RR^{d_1})}^2
		\leq
		2^{2(d_1+1)+\lambda}\norm{f(\cdot,y)}_{L^2(\RR^{d_1})}^2 \quad \text{ for all $y \in \RR^{d_2}$}
		.
	\]
	The claim now follows by integration over $y \in \RR^{d_2}$.
\end{proof}%

\begin{corollary}\label{cor:choice_Kc}
	Let $\lambda \geq 1$ and $C = 32d_1(1+\sqrt{\log\kappa})$.
	Then the subset $\cK_c := \{ k\in\cK \colon Q_k\cap (B(0,C\lambda^{1/2}) \times \RR^{d_2}) \neq \emptyset\}$ satisfies
	\[
		\sum_{k \in \cK_c^\complement} \norm{f}_{L^2(Q_k)}^2
		\leq
		\frac{1}{4}\norm{f}_{L^2(\RR^d)}^2
		\quad\text{for all}\quad
		f \in \Ran P_{(-\infty,\lambda]}(H_\cI)
		.
	\]
\end{corollary}

Here $B(0,C\lambda^{1/2})$ denotes the ball in $\RR^{d_1}$.

\begin{proof}
	For $f \in \Ran P_{(-\infty,\lambda]}(H_\cI)$ and $s \geq C\lambda^{1/2}$, Lemma~\ref{lem:decay} implies that
	\eqs{
		\norm{f}_{L^2(\RR^d\setminus (B(0,s)\times \RR^{d_2}))}^2
		& =
		\norm{\euler^{-|x_\cI|^2/64d_1}\euler^{|x_\cI|^2/64d_1}f}_{L^2(\RR^d\setminus (B(0,s)\times \RR^{d_2}))}^2\\
		& \leq
		\euler^{-s^2/32d_1}2^{2(d_1+1)+\lambda}\norm{f}_{L^2(\RR^d)}^2\\
		& \leq
		\frac{1}{4\kappa}\norm{f}_{L^2(\RR^d)}^2
		.
	}
	By definition, $Q_k\cap (B(0,CN^{1/2})\times\RR^{d_2})=\emptyset$ for  $k \in \cK_c^\complement$. Hence,
	\[
		\sum_{k \in \cK_c^\complement} \norm{f}_{L^2(Q_k)}^2
		\leq
		\kappa \norm{f}_{L^2(\RR^d\setminus (B(0,CN^{1/2})\times\RR^{d_2}))}^2
		\leq
		\frac{1}{4}\norm{f}_{L^2(\RR^d)}^2
		.\qedhere
	\]
\end{proof}%

The above motivates the following general hypothesis on the covering.

\begin{hypothesis}{$(\mathrm{H_\lambda})$}\label{hypothesis}
	Let $\cK$ be finite or countably infinite and let $(Q_k)_{k \in \cK}$ be an essential covering of $\RR^d$
	with multiplicity at most $\kappa$ as in \eqref{eq:covering-general-assumption}. Set $C = 32d_1(1 + \sqrt{\log\kappa})$.
	For fixed $\lambda \geq 1$, let
	\be\label{eq:defKc}
		\cK_c = \cK_c(\lambda) = \{ k \in \cK \colon Q_k \cap (B(0,C\lambda^{1/2})\times\RR^{d_2}) \neq \emptyset\}
	\ee
	For each $k\in\cK_c$, we suppose that
	\begin{enumerate}[(i)]
		\item
		$Q_k$ is non-empty, convex, open, and contained in a hyperrectangle with sides
		of length $l_k = (l_k^{(1)},\dots,l_k^{(d)}) \in (0,\infty)^d$
		parallel to the coordinate axes such that
		\item
		$\norm{l_k}_2 := \bigl((l_k^{(1)})^2 + \dots + (l_k^{(d)})^2 \bigr)^{1/2}\leq D \lambda^{(1-\eps)/2}$ for some $\eps \in (0,1]$ and $D > 0$
		independent of $k\in\cK_c$.
	\end{enumerate}
\end{hypothesis}

In what follows we call a set $Q \subset \RR^d$ \emph{centrally symmetric} if there is $x_0\in Q$ such that $x_0+x \in Q$ implies $x_0-x\in Q$.

Our general spectral inequality reads as follows.
Its proof is postponed to Subsection~\ref{ssec:proofSpecIneq} below.

\begin{theorem}\label{thm:gen}
	With fixed $\lambda \geq 1$ assume Hypothesis~\ref{hypothesis}.
	Let $a \geq 0$ and $\gamma \in (0,1)$ be given.
	If $\omega \subset \RR^d$ is measurable satisfying
	\be\label{eq:gen_assumption}
		\frac{\abs{\omega \cap Q_k}}{\abs{Q_k}}
		\geq
		\gamma^{\lambda^{a/2}}
		\quad\text{ for all }\
		k \in \cK_c
		,
	\ee
	then
	\be\label{eq:gen}
		\norm{f}_{L^2(\omega)}^2
		\geq
		\frac{3}{\kappa} \Biggl[ \frac{\gamma}{24 \cdot 2^d d^{1+d}} \Biggr]^{7\bigl( 1600\euler D(D+1) + \log(4\kappa^{1/2})\bigr)
		\lambda^{1-(\eps-a)/2}} \norm{f}_{L^2(\RR^d)}^2
		,
	\ee
	for every $f \in \Ran P_{(-\infty,\lambda]}(H_\cI)$.
	Here $\tau_d$ denotes the Lebesgue measure of the Euclidean unit ball in $\RR^d$.
	
	If, in addition, all $Q_k, k\in\cK_c$, are centrally symmetric
	(or a cube, respectively), then the term in square brackets can be replaced by
	\begin{equation}\label{eq:special-Q}
		\frac{\gamma}{24 \cdot 2^d d^{1+d/2}} \quad \left( \text{ or } \frac{\gamma}{24d^{1+d/2}\tau_d},
		\text{ respectively} \right).
	\end{equation}
\end{theorem}

\begin{remark}
	(a)
	Examples of families $(Q_k)_{k\in\cK}$ satisfying Hypothesis~\ref{hypothesis} are discussed in
	Section~\ref{ssec:examplesSpecIneq} below.

	(b)
	Let us emphasize that on one hand, $\eps$ and $D$ in condition
	(ii) need to be uniform in $k \in \cK_c$.
	On the other, formally they are allowed to depend on $\lambda$.
	However, in all applications presented in this paper this will not be the case implying that the exponent in \eqref{eq:gen}
	is proportional to $\lambda^{1-\frac{\eps-a}{2}}$. In this case the relevant
	power satisfies $1-\frac{\eps-a}{2} < 1$ if and only if $a < \eps$.
\end{remark}

\subsection{The local estimate and good covering sets}\label{sec:local-estimate}

On a bounded domain the following local estimate is sufficient to derive the type of uncertainty relation we are aiming at.
We rely here on crucial ideas of Nazarov~\cite{Nazarov-94} and Kovrijkine~\cite{Kovrijkine-thesis,Kovrijkine-01}.
They have been used and (at least implicitly) formulated in several recent works related to our topic,
such as~\cite[Section~5]{EgidiV-20},~\cite{WangWZZ-19},~\cite[Section~3.3.3]{BeauchardJPS-21},
\cite{MartinPS-20-arxiv}, and~\cite[Lemma~3.5]{EgidiS-21}.
We spell out the formulation from the last mentioned reference:

\begin{lemma}\label{lem:localEstimate}
	Let $\lambda \geq 1$, $f \in \Ran P_{(-\infty,\lambda]}(H_\cI)$, and let $Q \subset \RR^d$ be a non-empty bounded convex open
	set that is contained in a hyperrectangle with sides of length $l \in (0,\infty)^d$ parallel to coordinate axes.

	Then, for every measurable set $\omega \subset \RR^d$ and every linear bijection $\Psi \colon \RR^d \to \RR^d$ we have
	\bes
		\norm{f}_{L^2(Q \cap \omega)}^2
		\geq
		12 \Bigl( \frac{\abs{\Psi(Q \cap \omega)}}{24d\tau_d(\diam\Psi(Q))^d} \Bigr)^{4\frac{\log M}{\log 2}+1}
			\norm{f}_{L^2(Q)}^2
	\ees
	with
	\[
		M := \frac{\sqrt{\abs{Q}}}{\norm{f}_{L^2(Q)}} \cdot \sup_{z \in Q + D_{4l}} \abs{f(z)},
	\]
	where $D_{4l}\subset\CC^d$ denotes the polydisc of radius $4l$ centered at the origin.
\end{lemma}

Note that the normalized supremum $M$ in the above lemma automatically satisfies $M \geq 1$.
In order to estimate
\be\label{eq:distortion}
	\frac{\abs{\Psi(Q \cap \omega)}}{(\diam\Psi(Q))^d}
	=
	\frac{\abs{Q \cap \omega}}{\abs{Q}} \cdot \frac{\abs{\Psi(Q)}}{(\diam\Psi(Q))^d}
\ee
we may choose $\Psi = \Id$ in case of a cube and get
\be\label{eq:quotient-cube}
	\frac{\abs{\Psi(Q)}}{(\diam \Psi(Q))^d}
	=
	\frac{1}{d^{d/2}}
	.
\ee
For the general case we use the following corollary to John's Ellipsoid Theorem.

\begin{proposition}\label{prop:John}
	Let $\emptyset \neq Q\subset \RR^d$ be convex, open, and bounded.
	Then there is a linear bijection  $\Psi \colon \RR^d \to \RR^d$ with
	\be\label{eq:quotient-estimates}
		\eta
		:=
		\frac{\tau_d}{2^d d^d}
		\leq
		\frac{\abs{\Psi(Q)}}{(\diam \Psi(Q))^d}
		\leq
		\frac{\tau_d}{2^{d/2}}
		.
	\ee
	If, in addition, $Q$ is centrally symmetric,
	then $\eta$ can be replaced by $\tau_d/(4d)^{d/2}$.
\end{proposition}

\begin{proof}
	We first prove the upper bound in \eqref{eq:quotient-estimates}.
	By Jung's Theorem \cite{Jung-01}, $\Psi(Q)$ is contained in a
	ball $B$ of radius $R > 0$ satisfying
	\[
		R
		\leq
		\diam(\Psi(Q)) \sqrt{\frac{d}{2(d+1)}}
		\leq
		\diam(\Psi(Q))/\sqrt{2}
		.
	\]
	Hence we obtain
	\[
		\abs{\Psi(Q)}
		\leq
		\abs{B}
		=
		\frac{\tau_d}{2^{d/2}} (\diam\Psi(Q))^d
		.
	\]
	For the lower bound we use John's Theorem \cite{John-48}, which states that
	for every convex, open, bounded  $\emptyset\neq Q\subset\RR^d$ there is a linear
	bijection  $\Phi\colon \RR^d\to \RR^d$, some $z \in \RR^d$, and a radius $r>0$ such that the ellipsoid
	$\cE=\Phi(B_r)$ satisfies $\cE\subset Q +z \subset d \cdot \cE$ or, equivalently, setting $\Psi=\Phi^{-1}$
	\[
		B_r
		\subset
		\Psi(Q) +\Psi z
		\subset
		d \cdot B_r
		.
	\]
	This implies that $2r \leq \diam \Psi(Q) \leq 2r d$, as well as
	\[
		|\Psi(Q)|
		\geq
		\tau_d r^d
		\geq
		\tau_d \left(\frac{\diam(\Psi(Q))}{2d}\right)^d
		=
		\frac{\tau_d}{(2d)^d}\left(\diam(\Psi(Q))\right)^d.
	\]
	For centrally symmetric $Q$, John's Theorem gives
	$\cE\subset Q +z \subset \sqrt d \cdot \cE$ leading in the same
	way to the stated inequality.
\end{proof}

For the rest of this section, we fix $\lambda \geq 1$ and assume Hypothesis~\ref{hypothesis} for that $\lambda$.
Given a non-zero $f \in \Ran P_{(-\infty,\lambda]}(H_\cI)$, let
\be\label{eq:def_Mk}
	M_k := \frac{\sqrt{\abs{Q_k}}}{\norm{f}_{L^2(Q_k)}} \cdot \sup_{z \in Q_k + D_{4l_k}} \abs{f(z)}
\ee
denote the normalized supremum from the local estimate in Lemma~\ref{lem:localEstimate} corresponding to $Q_k$.
We do not know how to guarantee an upper bound on $M_k$ for all $k$, but for `sufficiently many' $k$. In order to make this
precise, we first derive for functions in $\Ran P_{(-\infty,\lambda]}(H_\cI)$ a so-called Bernstein-type inequality.
For the particular case of the harmonic oscillator, that is, for $\cI = \cJ = \{1,\dots,d\}$, this was first established in
\cite[Proposition~4.3\,(ii)]{BeauchardJPS-21} and later reproduced in a slightly different form in \cite[Proposition~B.1]{EgidiS-21}.

\begin{lemma}\label{lem:Bernstein}
	Given $\lambda \geq 1$, every function $f \in \Ran P_{(-\infty,\lambda]}(H_\cI)$ satisfies
	\begin{align*}
		\sum_{\abs{\alpha}=m} \frac{1}{\alpha!}\norm{\partial^\alpha f}_{L^2(\RR^d)}^2
		&\leq
		\frac{C_B(m,\lambda)}{m!} \cdot \norm{f}_{L^2(\RR^d)}^2
		\quad\text{ for all }\
		m \in \NN_0
		,
	\end{align*}
	where $\displaystyle C_B(m,\lambda):= 2^m \prod_{k=0}^{m-1}(\lambda+2k)$.
\end{lemma}

\begin{proof}
	Recall from \cite[Proposition~B.1]{EgidiS-21}
	and its proof that every $g \in \Ran P_{(-\infty,\lambda]}(H_1)$
	satisfies
	\[
		\sum_{\abs{\beta}=m} \frac{1}{\beta!}\norm{\partial^\beta g}_{L^2(\RR^{d_1})}^2
		\leq
		\frac{C_1(m,\lambda)}{m!} \norm{g}_{L^2(\RR^{d_1})}^2
		\quad\text{ for all }\
		m \in \NN_0
	\]
	with
	\[
		C_1(m,\lambda)
		=
		\prod_{k=0}^{m-1} (\lambda + 2k)
		.
	\]
	Moreover, every function $h \in \Ran P_{(-\infty,\lambda]}(H_2)$ satisfies
	\be \label{eq:Bernstein-H2}
		\sum_{\abs{\nu}=m} \frac{1}{\nu!}\norm{\partial^\nu h}_{L^2(\RR^{d_2})}^2
		\leq
		\frac{\lambda^m}{m!} \norm{h}_{L^2(\RR^{d_2})}^2
		\quad\text{ for all }\
		m \in \NN_0
	\ee
		which follows from \cite[Proposition~2.10]{EgidiS-21}; cf.~also \cite[Theorem~11.3.3]{Boas-54}.

	Let $f \in \Ran P_{(-\infty,\lambda]}(H_\cI)$. By Corollary~\ref{cor:specfamH}, we have that
	$(\partial^\beta f)(x,\cdot)$ belongs to $\Ran P_{(-\infty,\lambda]}(H_2)$ for all $x \in \RR^{d_1}$ and all
	$\beta \in \NN_{0,\cI}^d$. In the next step we split a multiindex $\alpha \in \NN_0^d$ as $\alpha = \beta + \nu$ with
	$\beta \in \NN_{0,\cI}^d$ and $\nu \in \NN_{0,\cI^\complement}^d$.
	Now we apply for fixed $m\in\NN$ and $\beta \in \NN_{0,\cI}^d$ inequality~\eqref{eq:Bernstein-H2}
	to $h=\partial^\beta f$ as well as Fubini's theorem to obtain
	\eqs{
		\sum_{\abs{\nu}=m-\abs{\beta}} \frac{1}{\nu!} \norm{\partial^\nu\partial^\beta f}_{L^2(\RR^d)}^2
		&\leq
		\frac{\lambda^{m-\abs{\beta}}}{(m-\abs{\beta})!} \norm{\partial^\beta f}_{L^2(\RR^d)}^2
		.
	}
	In the same way, $f(\cdot,y)$ belongs to $\Ran P_{(-\infty,\lambda]}(H_1)$ for all $y \in \RR^{d_2}$, so that
	\eqs{
		\sum_{\abs{\beta}=j} \frac{1}{\beta!} \norm{\partial^\beta f}_{L^2(\RR^d)}^2
		&\leq
		\frac{C_1(j,\lambda)}{j!} \norm{f}_{L^2(\RR^d)}^2
		.
	}
	Putting the last two estimates together, we arrive at
	\eqs{
		\sum_{\abs{\alpha} = m} \frac{1}{\alpha!} \norm{\partial^\alpha f}_{L^2(\RR^d)}^2
		&=
		\sum_{j=0}^m \sum_{\abs{\beta} = j} \frac{1}{\beta!} \sum_{\abs{\nu}=m-j} \frac{1}{\nu!}
			\norm{\partial^\nu\partial^\beta f}_{L^2(\RR^d)}^2\\
		&\leq
		\frac{1}{m!} \biggl( \sum_{j=0}^m \binom{m}{j} C_1(j,\lambda) \lambda^{m-j} \biggr) \norm{f}_{L^2(\RR^d)}^2
		.
	}
	In order to complete the proof, it only remains to observe that
	\eqs{
		\sum_{j=0}^m \binom{m}{j} C_1(j,\lambda) \lambda^{m-j}
		&=
		\sum_{j=0}^m \binom{m}{j} \prod_{k=0}^{j-1} (\lambda + 2k) \cdot \lambda^{m-j}\\
		&\leq
		\prod_{k=0}^{m-1} (\lambda + 2k) \sum_{j=0}^m \binom{m}{j}\\
		&=
		C_1(m,\lambda)  2^m
		.\qedhere
	}
\end{proof}%

Kovrijkine~\cite{Kovrijkine-thesis,Kovrijkine-01} established the approach of localizing the Bern\-stein-type inequality on so-called good $Q_k$.
It was used in many works thereafter, e.g.~\cite{EgidiV-20,BeauchardJPS-21}.
We rely here on the form presented in~\cite[Section~3.3]{EgidiS-21}:

We say that $Q_k$ for $k \in \cK$ is \emph{good} with respect to $f \in \Ran P_{(-\infty,\lambda]}(H_\cI)$ if
\[
	\sum_{\abs{\alpha} = m} \frac{1}{\alpha!}\norm{\partial^\alpha f}_{L^2(Q_k)}^2
	\leq
	2^{m+1} \kappa \frac{C_B(m,\lambda)}{m!} \norm{f}_{L^2(Q_k)}^2
	\quad\text{ for all }\
	m \in \NN,
\]
and we call $Q_k$ \emph{bad} otherwise. We then have
\be\label{eq:bad-mass}
	\sum_{k \in\cK\colon Q_k\text{ bad}} \norm{f}_{L^2(Q_k)}^2
	\leq
	\frac{1}{2}\norm{f}_{L^2(\RR^d)}^2
\ee
and we set
\be\label{eq:defKg}
	\cK_g
	:=
	\{ k \in \cK \colon Q_k\text{ good} \}
	.
\ee
Inequality \eqref{eq:bad-mass} shows that the $Q_k$ with $k\in\cK_g$ carry at least half of the $L^2$-mass of $f$.
However, we actually need a similar statement with $k\in\cK_g$ replaced by the intersection $\cK_c \cap \cK_g$.
This is guaranteed by the following lemma.

\begin{lemma}\label{lem:reducCovering}
	Given $f \in \Ran P_{(-\infty,\lambda]}(H_\cI)$ and $\cK_c$ and $\cK_g$ as in \eqref{eq:defKc} and \eqref{eq:defKg},
	respectively, we have
	\[
		\norm{f}_{L^2(\RR^d)}^2
		\leq
		4 \sum_{k \in \cK_c \cap \cK_g} \norm{f}_{L^2(Q_k)}^2
		.
	\]
	In particular,  $\cK_c \cap \cK_g \neq \emptyset$.
\end{lemma}

\begin{proof}
	Subadditivity, Corollary~\ref{cor:choice_Kc}, and \eqref{eq:bad-mass} imply that
	\[
		\sum_{k \in \cK_c^\complement \cup \cK_g^\complement} \norm{f}_{L^2(Q_k)}^2
		\leq
		\sum_{k \in \cK_c^\complement} \norm{f}_{L^2(Q_k)}^2 + \sum_{k \in \cK_g^\complement} \norm{f}_{L^2(Q_k)}^2
		\leq
		\frac{3}{4} \norm{f}_{L^2(\RR^d)}^2
		.
	\]
	Passing to the complementary sum  over $k \in \cK_c \cap \cK_g$  proves the claim.
\end{proof}%

A key ingredient in the strategy of Kovrijkine \cite{Kovrijkine-01,Kovrijkine-thesis} is the observation that
each good $Q_k$ contains a point where a Taylor expansion with suitable upper bounds on the coefficients can be performed.
This can be proven by contradiction, see \cite[(1.5)]{Kovrijkine-thesis}.
Technically we follow the presentation in \cite{EgidiS-21} and show that for each
$k \in \cK_c \cap \cK_g$ there is a point $x_k \in Q_k$ with
\be\label{eq:pointwiseLocalBernstein}
	\sum_{|\alpha| = m} \frac{1}{\alpha!} \abs{\partial^\alpha f(x_k)}^2
	\leq
	\frac{4^{m+1} \kappa C_B(m,\lambda)}{m!}\frac{\norm{f}_{L^2(Q_k)}^2}{\abs{Q_k}}
\ee
for all $m \in \NN_0$ and all $\alpha \in \NN_0^d$ with $\abs{\alpha} = m$,
see \cite[Eq.~(3.9)]{EgidiS-21}. In order to see
this, we assume for contradiction that for all $x \in Q_k$ there is $m = m(x) \in \NN_0$ with
\[
	\sum_{\abs{\alpha}=m} \frac{1}{\alpha!} \abs{\partial^\alpha f(x)}^2
	>
	\frac{4^{m+1}\kappa C_B(m,\lambda)}{m!\abs{Q_k}} \norm{f}_{L^2(Q_k)}^2
	.
\]
We multiply the latter by $m!4^{-m-1}/(\kappa C_B(m,\lambda))$, estimate further by taking the sum over all $m \in \NN_0$ on the
left-hand side, integrate over $Q_k$, and take into account that $Q_k$ is good to obtain
\[
	\norm{f}_{L^2(Q_k)}^2
	<
	\norm{f}_{L^2(Q_k)}^2 \sum_{m\in\NN_0} 2^{-m-1}
	=
	\norm{f}_{L^2(Q_k)}^2
	,
\]
leading to a contradiction.
This proves \eqref{eq:pointwiseLocalBernstein}.

Using Taylor expansion around $x_k$, we now obtain similarly as in the
proof of \cite[Proposition~3.1]{EgidiS-21} the following result.

\begin{lemma}\label{lem:Mk}
	Let $k \in \cK_c \cap \cK_g$.
	Then, the quantity $M_k$ in~\eqref{eq:def_Mk} satisfies
	\bes
		M_k
		\leq
		2\kappa^{1/2}\sum_{m \in \NN_0} C_B(m,\lambda)^{1/2}
		 \frac{(10 \norm{l_k}_2)^m}{m!}
		,
	\ees
	where $\norm{l_k}_2^2 = (l_k^{(1)})^2 + \dots + (l_k^{(d)})^2$.
\end{lemma}

\begin{proof}
	Let $x_k \in Q_k$ be a point as in \eqref{eq:pointwiseLocalBernstein}. Using Taylor expansion of $f$ around $x_k$, for every
	$z \in x_k + D_{5l_k}$ we then have
	\eqs{
		\abs{f(z)}
		&\leq
		\sum_{\alpha \in \NN_0^d} \frac{\abs{\partial^\alpha f(x_k)}}{\alpha!} \abs{(z-x_k)^\alpha}
		\leq
		\sum_{m\in\NN_0}\sum_{\abs{\alpha}=m} \frac{\abs{\partial^\alpha f(x_k)}}{\alpha!} (5l_k)^\alpha\\
		&\leq
		\sum_{m\in\NN_0} \biggl(\sum_{\abs{\alpha}=m} \frac{(5l_k)^{2\alpha}}{m!}\biggr)^{1/2}
			\biggl(\sum_{\abs{\alpha}=m} \frac{\abs{\partial^\alpha f(x_k)}^2}{\alpha!}\biggr)^{1/2}\\
		&\leq
		\sum_{m\in\NN_0} \frac{(5\norm{l_k}_2)^m}{\sqrt{m!}}
		\biggl(\sum_{\abs{\alpha}=m} \frac{\abs{\partial^\alpha f(x_k)}^2}{\alpha!}\biggr)^{1/2}\\
		&\leq
		2\kappa^{1/2} \frac{\norm{f}_{L^2(Q_k)}}{\sqrt{\abs{Q_k}}} \sum_{m\in\NN_0} C_B(m,\lambda)^{1/2} \frac{(10\norm{l_k}_2)^m}{m!}
		,
	}
	where for the second last inequality we used that $\sum_{|\nu| = m} l_k^{2\nu}/\nu! = \norm{l_k}_2^{2m}/m!$.
	Taking into account that $Q_k + D_{4l_k} \subset x_k + D_{5l_k}$, this proves the claim.
\end{proof}%

\subsection{Proof of Theorem~\ref{thm:gen}}\label{ssec:proofSpecIneq}

With the above preparations, we are finally in position to prove our abstract spectral inequality.

\begin{proof}[Proof of Theorem~\ref{thm:gen}]
	In light of Hypothesis~\ref{hypothesis}, the local estimate in
	Lemma~\ref{lem:localEstimate} and
	Proposition~\ref{prop:John} yield
	\[
		\norm{f}_{L^2(Q_k\cap \omega)}^2
		\geq
		a_k\norm{f}_{L^2(Q_k)}^2
		\quad\text{ with}\quad
		a_k
		=
		12 \Bigl( \frac{\abs{Q_k \cap \omega}}{24\cdot 2^d d^{1+d} \abs{Q_k}} \Bigr)^{4\frac{\log M_k}{\log 2} + 1}
	\]
	for $k \in \cK_c$, where $M_k$ is as in~\eqref{eq:def_Mk}.
	By Lemma~\ref{lem:reducCovering} we then have
	\be\label{eq:gen:min}
		\begin{split}
			\Bigl( \min_{k \in \cK_c \cap \cK_g} a_k \Bigr) \norm{f}_{L^2(\RR^d)}^2
			&\leq
			4\sum_{k \in \cK_c \cap \cK_g} a_k \norm{f}_{L^2(Q_k)}^2\\
			&\leq
			4\sum_{k \in \cK_c \cap \cK_g} \norm{f}_{L^2(Q_k\cap \omega)}^2
			\leq
			4\kappa \norm{f}_{L^2(\omega)}^2
			.
		\end{split}
	\ee
	Using assumption~\eqref{eq:gen_assumption} on the set $\omega$, we have
	\be\label{eq:gen:ak}
		a_k
		\geq
		12\Bigl( \frac{\gamma^{\lambda^{a/2}}}{24\cdot 2^d d^{1+d} } \Bigr)^{4\frac{\log M_k}{\log 2}+1}
		\quad\text{ for all}\
		k \in \cK_c
		.
	\ee
	In order to proceed further, we recall that condition (ii)
	of Hypothesis~\ref{hypothesis}
	gives $\norm{l_k}_2 \leq D\lambda^{(1-\eps)/2}$ for all $k \in \cK_c$ and infer from the proof of
	Proposition B.1 in \cite{EgidiS-21} that
	\[
		\prod_{k=0}^{m-1} (\lambda + 2k)
		\leq
		(2\delta)^{2m} \euler^{\euler/\delta^2} (m!)^2 \euler^{2\sqrt{\lambda}/\delta} \quad \text{ for } \delta >0.
	\]
	Hence, Lemma~\ref{lem:Mk} and the definition of $C_B(m,\lambda)$ in Lemma~\ref{lem:Bernstein} imply
	\eqs{
		M_k
		&\leq
		2\kappa^{1/2} \sum_{m \in \NN_0} C_B(m,\lambda)^{1/2} \frac{(10 D\lambda^{(1-\eps)/2})^m}{m!}\\
		&=
		2\kappa^{1/2} \euler^{\euler/(2\delta^2)} \euler^{\sqrt{\lambda}/\delta} \sum_{m \in \NN_0} (20\sqrt{2}\delta D\lambda^{(1-\eps)/2})^m
		\quad \text{ for all } k \in \cK_c \cap \cK_g
	}
	for $\delta>0$.
	Choosing
	\bes
		\delta
		=
		\bigl( 40\sqrt{2}D\lambda^{(1-\eps)/2} \bigr)^{-1}
		,
	\ees
	we obtain
	\bes
		\begin{split}
		M_k &\leq 4\kappa^{1/2}\exp(1600\euler D^2 \lambda^{1-\eps} + 40\sqrt{2}D\lambda^{(1-\eps)/2}\sqrt{\lambda})\\
		&\leq 4\kappa^{1/2}\exp(1600\euler D(D+1) \lambda^{1-\eps/2})
		\end{split}
	\ees
	and, thus,
	\eqs{
		\log M_k
		&\leq
		\log(4\kappa^{1/2})+1600\euler D(D+1) \lambda^{1-\eps/2} \\
		&\leq
		\bigl( 1600\euler D(D+1) + \log(4\kappa^{1/2}) \bigr) \lambda^{1-\eps/2}
	}
	for all $k \in \cK_c \cap \cK_g$.
	Combining the latter with \eqref{eq:gen:ak}, we arrive at
	\[
		a_k
		\geq
		12\Bigl( \frac{\gamma}{24 \cdot 2^d d^{1+d}} \Bigr)^{7\bigl( 1600\euler D(D+1) + \log(4\kappa^{1/2})\bigr) \lambda^{1-(\eps-a)/2}}
	\]
	for all $k \in \cK_c \cap \cK_g$,
	where we used that $1+4/\log2 \leq 7$. In view of \eqref{eq:gen:min}, this proves the claim.

	If all $Q_k, k\in\cK_c$ are centrally symmetric we use the sharper lower bound in Proposition~\ref{prop:John} to
	replace $d^{1+d}$ by $d^{1+d/2}$
	in the lower bound on $a_k$, and similarly in the case of cubes.
\end{proof}%

\subsection{Examples}\label{ssec:examplesSpecIneq}

We now discuss examples of sets $\omega \subset \RR^d$, where Theorem~\ref{thm:gen}
can be applied with $D$ and $\eps$ not depending on $\lambda$. In the situation of Theorem~\ref{thm:ucp}, these sets are characterized in terms of an
explicit covering, but for Corollary~\ref{cor:specIneq:genSets} below the covering is implicitly constructed using Besicovitch's
covering theorem. Both results should be regarded as corollaries to Theorem~\ref{thm:gen}.

We start with the proof of Theorem~\ref{thm:ucp}.

\begin{proof}[Proof of Theorem~\ref{thm:ucp}]
	Take $Q_k = \Lambda_L(k)$ for $k \in \cK = (L\ZZ)^d$. We then have $\kappa = 1$ and, thus, $C = 32d_1 \leq 32d$ in Hypothesis \ref{hypothesis}.
	In view of the asymptotic formula $\tau_d \sim (2\pi\euler/d)^{d/2} / \sqrt{d\pi}$
	we infer the bound $24d^{1+d/2}\tau_d \leq K^d$ for the term appearing in
	\eqref{eq:special-Q} in the case of a cube.

	It is easy to see that $l_k = (L,\dots,L)$ satisfies
	$\norm{l_k}_2 = \sqrt{d}L = D\lambda^{0}$ with $D := \sqrt{d}L$. Hence,
	$(Q_k)_{k\in\cK}=(\Lambda_L(k))_{k\in(L\ZZ)^d}$ satisfies Hypothesis~\ref{hypothesis} for every $\lambda \geq 1$.

	It is also not hard to verify that
	\[
		\frac{|k_\cI|}{2}
		\leq
		\inf_{x \in \Lambda_L(k)}|x_\cI|
		\leq
		C\lambda^{1/2}
		\quad\text{for all}\
		k\in\cK_c\subset (L\ZZ)^d.
	\]
	Here, the first inequality follows from the definition of $\Lambda_L(k)$ while the second follows from the definition of
	$\cK_c$. Finally, using these estimates, we calculate
	\[
		\gamma^{1+\abs{k_\cI}^a}
		\geq
		\bigl( \gamma^{2^a} \bigr)^{1+(\abs{k_\cI}/2)^a}
		\geq
		\bigl( \gamma^{2^a} \bigr)^{1+{C}^a \lambda^{a/2}}
		\geq
		\bigl(\gamma^{2(2C)^a}\bigr)^{\lambda^{a/2}}
		.
	\]
	The claim in Theorem~\ref{thm:ucp} now follows from Theorem~\ref{thm:gen} with $\eps = 1$ and
	$\gamma$ replaced by $\gamma^{2(2C)^a}$. It only remains to observe the particular constant in \eqref{eq:specIneq-fixed-cubes}
	from the simple estimate
	\[
		2\cdot (2C)^a\cdot 7 \bigl( 1600\euler D(D+1) + \log(4)\bigr)
		\leq
		K d^{1+a}(1+L)^2
		.
		\qedhere
	\]
\end{proof}%

We may also consider sets $\omega$ with respect
to a scale that is allowed to vary in the coordinate directions corresponding to $\cI$. To this end, let
$\rho\colon\RR^{d_1}\to (0,\infty)$ be any function that satisfies
\[
	\rho(x)\leq R(1+|x|^2)^{\frac{1-\eps}{2}}
	\quad\text{for all}\
	x \in \RR^{d_1}
\]
with $R>0$ and $\eps\in(0,1]$, and let $L > 0$.
Given $x \in \RR^d$, we introduce the coordinates $x = (x^{(1)},x^{(2)}) \in \RR^{d_1} \times \RR^{d_2}$ and set
\[
	Q(x)
	:=
	B(x^{(1)},\rho(x^{(1)})) \times \Lambda_{L}(x^{(2)})
	\subset
	\RR^{d_1} \times \RR^{d_2}
	.
\]

The following result now generalizes \cite[Theorem~2.1]{MartinPS-20-arxiv} and \cite[Theorem~2.7]{DickeSV-21}.

\begin{corollary}\label{cor:specIneq:genSets}
	Let $\omega\subset\RR^d$ be a measurable set with
	\[
		\frac{\abs{\omega\cap Q(x)}}{\abs{Q(x)}}
		\geq
		\gamma^{1+|x_\cI|^a}
		\quad\text{for all}\
		x \in \RR^d
	\]
	and for some fixed $a\in [0,\eps)$, $\gamma \in (0,1)$.
	
	Then, there is a universal constant $K\geq 1$ such that for every $\lambda \geq 1$ and all $f \in \Ran P_{(-\infty,\lambda]}(H_\cI)$
	we have
	\be\label{eq::specIneq:genSets}
		\norm{f}_{L^2(\omega)}^2
		\geq
		3\Bigl(\frac{\gamma}{\euler}\Bigr)^{K^{1+a}d^{(13+3a)/2}(1+R+L)^2 \lambda^{1-\frac{\eps-a}{2}}}\norm{f}_{L^2(\RR^d)}^2
		.
	\ee
\end{corollary}

In contrast to the situation in Theorem~\ref{thm:ucp}, the proof of Corollary~\ref{cor:specIneq:genSets} starts with the construction
of the family $(Q_k)_{k\in\cK}$, as the family is this time not given explicitly in the statement of the result. To this end, we use
the following formulation of the well-known Besicovitch covering theorem.

\begin{proposition}[Besicovitch]\label{prop:Besicovitch}
	If $A\subset\RR^{d_1}$ is a bounded set and $\cB$ is a family of closed balls such that each point in $A$ is the center of some ball
	in $\cB$, then there are at most countably many balls $(\overline{B}_k)\subset\cB$ such that
	\[
		\indic_A
		\leq
		\sum_k\indic_{\overline{B}_k}\leq K_{\mathrm{Bes}}^{d_1}
		,
	\]
	where $K_{\mathrm{Bes}} \geq 1$ is a universal constant.
\end{proposition}

\begin{proof}[Proof of Corollary~\ref{cor:specIneq:genSets}]
	Suppose first that $d_1 < d$.
	For fixed $\lambda \geq 1$, we consider the set $A = B(0,C\lambda^{1/2}) \subset \RR^{d_1}$, where $C=32d_1(1+\sqrt{\log(K_{\mathrm{Bes}}^{d_1})})$.
	Then, the assumptions of Proposition~\ref{prop:Besicovitch} are fulfilled for $A$ and the family of balls
	$\cB=\{\overline{B(x,\rho(x))}\colon x\in A\}$. This shows that there is a subset $\cK_* \subset\NN$ and a collection of points
	$(y_j)_{j\in\cK_*} \subset A$ such that the balls $B_j=B(y_j,\rho(y_j))$ satisfy $\abs{A\setminus\bigcup_{j\in\cK_*}B_j} = 0$.
	Setting $B_0 = \RR^{d_1}\setminus\bigcup_{j\in\cK_*}B_j$, the family $(B_j)_{j\in\cK_0}$, $\cK_0=\cK_*\cup\{0\}$, is then an essential
	covering of $\RR^{d_1}$ with
	\[
		\sum_{j\in\cK_0}\indic_{B_j}
		\leq
		K_{\mathrm{Bes}}^{d_1}
		=:
		\kappa
		.
	\]
	Set $\cK := \cK_0 \times (L\ZZ)^{d_2}$ and $Q_k := B_{k^{(1)}} \times \Lambda_{L}(k^{(2)})$	for
	$k = (k^{(1)},k^{(2)}) \in \cK$. Then, $(Q_k)_{k \in \cK}$ is an essential covering of $\RR^d$ with
	\[
		\sum_{k\in\cK}\indic_{Q_k}
		\leq
		\kappa
		.
	\]
	Note also that by construction we have
	\[
		\cK_c = \{ k \in \cK \colon Q_k \cap (B(0,C\lambda^{1/2}) \times \RR^{d_2}) \neq \emptyset \} = \cK_* \times (L\ZZ)^{d_2}
	\]
	and $Q_k = Q((y_{k^{(1)}},k^{(2)}))$ for $k \in \cK_c$.

	We show that $(Q_k)_{k \in \cK}$ satisfies Hypothesis~\ref{hypothesis}:
	It is easy to see that (i) is satisfied with $l_k = (2\rho(y_{k^{(1)}}),\dots,2\rho(y_{k^{(1)}}),L,\dots,L)$.
	Since $y_{k^{(1)}} \in A$ for all $k\in\cK_c$, we have $|y_{k^{(1)}}| \leq C\lambda^{1/2}$ and, consequently,
	\[
		\rho(y_{k^{(1)}})
		\leq
		2RC\lambda^{(1-\eps)/2}
		\quad\text{for all}\
		k\in\cK_c
		.
	\]
	Combining this with the identity for $l_k$ stated above, we obtain
	\[
		\norm{l_k}_2
		\leq
		\norm{l_k}_1
		\leq
		2d_1\rho(y_{k^{(1)}}) + d_2L
		\leq
		D\lambda^{(1-\eps)/2}
		,\quad
		D=d(4RC + L)
		.
	\]
	This proves condition (ii). Thus, Hypothesis~\ref{hypothesis} is satisfied.
	Using again $|y_{k^{(1)}}| \leq C\lambda^{1/2}$ for $k\in\cK_c$, we see that the assumption on the set $\omega$ yields
	\[
		\frac{|\omega\cap Q_k|}{|Q_k|}
		\geq
		\gamma^{1+(C\lambda^{1/2})^a}
		\geq
		\bigl( \gamma^{1+C^a} \bigr)^{\lambda^{a/2}}.
	\]
	We now apply Theorem~\ref{thm:gen} with $\gamma$ replaced by $\gamma^{1+C^a}$.
	Since the $Q_k$ are centrally symmetric for every $k \in \cK_c$, this gives
	\[
		\norm{f}_{L^2(\omega)}^2\geq \frac{3}{\kappa}\Biggl[\frac{\gamma^{1+C^a}}{24\cdot 2^d d^{1+d/2}}\Biggr]^{7\bigl( 1600\euler D(D+1) + \log(4\kappa^{1/2})\bigr)
		\lambda^{1-(\eps-a)/2}}\norm{f}_{L^2(\RR^d)}^2
		.
	\]
	For some appropriately chosen, universal constant $K \geq 1$
	we have $24\cdot 2^d d^{1+d/2} \leq (Kd)^d$, $\kappa \leq K^d$, $1+{C}^a \leq  (1 + K^a) d^{3a/2}$, and
	$D \leq Kd^{5/2}(R+L)$.
	Hence, possibly adapting the constant $K$, it is easy to see that
	\eqs{
		\bigl(\log(\kappa)+\log(24\cdot 2^d d^{1+d/2})\bigr)\cdot(1+{C}^a)\cdot 7\bigl( 1600\euler D(D+1) &+ \log(4\kappa^{1/2})\bigr)
		\\
		&\leq
		K^{1+a} d^{(13+3a)/2} (1+R + L)^2
		,
	}
	which gives the precise constant in the statement.

	If $d_1 = d$, then $d_2=0$ and the second factors in the tensor sets are empty.
	In this case the proof is similar, but even simpler.
\end{proof}%

\subsection{Spectral inequalities with parts of free potential}

Via Fourier transform we can reduce the
more general case $\cI\setminus\cJ\neq\emptyset$ to the previously studied situation $\cJ = \{1,\dots,d\}$.
More precisely, suppose that $\cI\cup\cJ = \{1,\dots,d\}$ while $\cJ \neq \{1,\dots,d\}$.
Then the partial Fourier transform
\[
	(\cF_{\cI\setminus\cJ} f)(x)
	=
	\frac{1}{(2\pi)^{m/2}}
	\int_{\RR^d_{\cI\setminus\cJ}} f(\eta,x^{(2)}) \euler^{-i \eta\cdot x^{(1)}}\Diff{\eta}
\]
where $m = \#(\cI\setminus\cJ)$ and $x = (x^{(1)},x^{(2)})$ with
$x^{(1)} \in \RR^d_{\cI\setminus\cJ}$ and $x^{(2)} \in \RR^d_{(\cI\setminus\cJ)^\complement}$
satisfies
\[
	H_{\cI,\cJ}
	=
	\cF_{\cI\setminus\cJ}^{-1} H_{\cI\cap\cJ,\cI\cup\cJ}\cF_{\cI\setminus\cJ}
	=
	\cF_{\cI\setminus\cJ}^{-1} H_{\cI\cap\cJ}\cF_{\cI\setminus\cJ}
	,
\]
see Lemma~\ref{lem:operators-and-partial-ft}.
Thus spectral inequalities of the form $\norm{\indic_\omega f}_{2}^2\geq C \norm{f}_{2}^2$ for $H_{\cI\cap\cJ}$
translate directly to spectral inequalities for $H_{\cI,\cJ}$ of the form $\norm{Bf}_{2}^2\geq C \norm{f}_{2}^2$
with $B = \cF_{\cI\setminus\cJ}^{-1}\indic_\omega\cF_{\cI\setminus\cJ}$ and the same constant $C > 0$.
Since $H_{\cI\cap\cJ}$ is an operator of the form discussed
in the previous parts of Section~\ref{sec:spectral-inequality}, we get
analogous results for $H_{\cI,\cJ}$.
This is exemplified in the following result for the situation of Theorem~\ref{thm:ucp}.

\begin{corollary}\label{cor:spectral-inequality-free-potential}
	Suppose $\cI\cup\cJ = \{1,\dots,d\}$, $\cI\setminus\cJ\neq\emptyset$, and
	let $\omega$ be as in \eqref{eq:control-set-fixed-cubes}.
	Then, there is a universal constant $K\geq 1$ such that for every $\lambda\geq 1$
	and all $f\in \Ran P_{(-\infty,\lambda]}(H_{\cI,\cJ})$ we have
	\bes
		\norm{Bf}_{L^2(\RR^d)}^2
		\geq
		3\Bigl( \frac{\gamma}{K^d} \Bigr)^{K d^{1+a} (1+\rho)^2\lambda^{(1+a)/2}}
		\norm{f}_{L^2(\RR^d)}^2
		,
	\ees
	where $B = \cF_{\cI\setminus\cJ}^{-1}\indic_\omega\cF_{\cI\setminus\cJ}$.
\end{corollary}

Note that the case $\cI\cup\cJ \neq \{1,\dots,d\}$
can be reduced to the present case provided the sensor sets are chosen as appropriate Cartesian products.

\begin{remark}
	If $\omega$ is Borel measurable, then $B = \cF_{\cI\setminus\cJ}^{-1}\indic_\omega\cF_{\cI\setminus\cJ}$ can be interpreted by functional calculus.
	To this end, let $X_1,\dots,X_d$ be the strongly commuting position operators $X_j f = x_j f$.
	Then the multiplication operator $\indic_\omega$ agrees with $\indic_\omega(X_1,\dots,X_d)$ defined by
	joint functional calculus, cf.~\cite[Chapter~5.5]{Schmuedgen-12}.
	Since the momentum operators $P_1,\dots,P_d$ with $P_j f = -\ii\partial_j f$ correspond to the position operators by
	$\cF_{\{j\}}^{-1} X_j \cF_{\{j\}} = P_j$, we have
	\[
		B
		=
		\indic_\omega(R_1,\dots,R_d)
		,
		\quad \text{where} \quad
		R_j = \begin{cases}
			X_j, & j\in (\cI\setminus\cJ)^\complement\\
			P_j, & j\in \cI\setminus\cJ
		\end{cases}
		.
	\]
	Note that $R_1,\dots,R_d$ are likewise strongly commuting.
\end{remark}
%
%
\section{Applications, extensions, and comparison to previous results}\label{sec:applications-extensions}

In this section we give examples of quadratic differential operators that fit into our general framework.
These examples also present the application of the theory and extend previous results
by different authors, see, e.g., \cite{BeauchardPS-18,BeauchardJPS-21,MartinPS-20-arxiv,Alphonse-20,DickeSV-21}.

Moreover, we give a sharper version of Theorem~\ref{thm:observability} in the case where the operator $-A$ is itself a partial harmonic oscillator.
Furthermore, we also establish a result for isotropic Shubin operators, which are non-quadratic differential operators.

\subsection{Partial harmonic oscillators}

In the situation of the partial harmonic oscillator itself, we do not rely on Theorem~\ref{thm:obs_and_control} above.
Instead, we can use the stronger result from \cite{NakicTTV-20}, since, as already noted in the introduction, the dissipation estimate
is trivial for all times $t > 0$ if we choose the projections onto the spectral subspace of the non-negative self-adjoint operator for
which we want to establish observability or, equivalently, null-controllability.
This is again demonstrated for $\omega$ as in Theorem~\ref{thm:ucp}, but with $a=1/2$ for simplicity.
The theorem is a direct consequence of \cite[Theorem~2.8]{NakicTTV-20}.

\begin{theorem}[Observability for the partial harmonic oscillator]\label{thm:observability-sharp-control-costs}
	Let $H_\cI = -\Delta + |x_\cI|^2$ for $\cI\subset\{1,\dots,d\}$ be the partial harmonic oscillator and
	let $\omega$ be as in \eqref{eq:control-set-fixed-cubes} above with $a=1/2$.
	Then the abstract Cauchy problem \eqref{eq:controlled-abstract-Cauchy-problem} with $A = -H_\cI$ and $B = \indic_\omega$
	is observable and
	\[
		C_{\mathrm{obs}}^2
		\leq
		\frac{c_1}{T}\exp\Bigl( \frac{c_2}{T^{3}} (1+\rho)^{8}\gamma^{-4} \Bigr)
		,
	\]
	where $c_1,c_2>0$ depend only on the dimension $d$.
\end{theorem}

\begin{remark}
	Following Remark~\ref{rem:fractional-operator} we get an analogous result for the fractional harmonic oscillator $H_\cI^\theta$, $\theta>3/4$.
\end{remark}

Note that here $C_{\mathrm{obs}} \in \mathcal{O}(1/T^{1/2})$ as $T\to\infty$ whereas Theorem~\ref{thm:observability}
merely establishes $C_{\mathrm{obs}} \in \mathcal{O}(1)$ as $T\to\infty$.
Although a result similar to \cite{NakicTTV-20} has been obtained in the non-self-adjoint case \cite{GallaunST-20}, this is
not applicable in the situation of Theorem~\ref{thm:observability}: our dissipation estimate for general quadratic differential operators $A$ with
$S(A) = S(H_\cI)$ only holds for small times $t < 1$ but \cite{GallaunST-20} requires the dissipation estimate to hold for all times $t \leq T/2$.
In particular, this shows that the dissipation estimate from Theorem~\ref{thm:dissipation} is not optimal in the particular
case where $-A$ is a partial harmonic oscillator.

\subsection{Quadratic differential operators with zero singular space}\label{ssec:zero-singular-space}

If $q$ is any complex quadratic form with singular space $S(q) = \{0\}$ satisfying $\Re q \leq 0$, then Theorem~\ref{thm:dissipation}
holds for the projections onto the spectral subspace of the harmonic oscillator.
This situation has already been considered in \cite{BeauchardJPS-21,MartinPS-20-arxiv} based on the dissipation estimate from
\cite{BeauchardPS-18}. In this setting, the spectral inequality takes the form of a \emph{Logvinenko-Sereda type inequality} for
Hermite functions up to a given degree. The latter was first established in \cite{BeauchardJPS-21} for thick sensor sets $\omega$
and subsequently generalized to not necessarily thick sets in \cite{MartinPS-20-arxiv,DickeSV-21}.
For $\omega$ as in  \eqref{eq:decay-twosided} we obtain the following result, which is on one hand
a particular case of Corollary~\ref{cor:control} and on the other generalizes the just mentioned references.

\begin{corollary}
	Let $ S(A) = \{0\}$.
	There are sensor sets of finite measure such that the abstract Cauchy problem \eqref{eq:controlled-abstract-Cauchy-problem} is
	null-controllable.
\end{corollary}

Sensor sets as in the above corollary were not accessible before in this context.

\subsection{Null-controllability for isotropic Shubin operators}

Besides the above studied quadratic differential operators our results can also be used to extend the observability from \cite{Alphonse-20b}
for the semigroup generated by the negative of the isotropic Shubin operator
$(-\Delta)^k + |x|^{2k}$ for $k\in\NN$.
In fact, combining our spectral inequality
in Theorem~\ref{thm:ucp} with the dissipation estimate from the proof of
\cite[Theorem~2.8]{Alphonse-20b} shows, amongst others, the following result, where $\omega$ is again as in \eqref{eq:decay-twosided}.

\begin{corollary}
	For all $k\in\NN$ and every sensor set as in \eqref{eq:decay-twosided}
	the abstract Cauchy problem \eqref{eq:controlled-abstract-Cauchy-problem}
	with $A = -(-\Delta)^k - |x|^{2k}$ is null-controllable.
\end{corollary}
See also \cite{Martin-22} and \cite{DickeSV-22b} for very recent related results.

\subsection{Generalized Ornstein-Uhlenbeck operators}

For symmetric, positive semidefinite matrices $Q,R\in\RR^{d\times d}$ and a matrix $B\in\RR^{d\times d}$, we consider the generalized Ornstein-Uhlenbeck operators
\[
	A
	=
	\frac{1}{2}\mathrm{Tr}(Q \nabla_x^2) - \frac{1}{2} Rx\cdot x + Bx\cdot \nabla_x
\]
on $L^2(\RR^d)$, where $\mathrm{Tr}$ denotes the trace and $\nabla_x^2= (\partial_{x_j}\partial_{x_k})_{j,k=1}^d$.
We here assume without loss of generality that $\mathrm{Tr}(B) = 0$, as $A$ is in this case a quadratic differential operator
with symbol
\[
	q(x,\xi)
	=
	-\frac{1}{2}Q\xi\cdot\xi -\frac{1}{2}Rx\cdot x - \ii Bx\cdot\xi
	.
\]
Note that without the assumption on the trace, the associated semigroup changes by the constant factor $\exp(-\mathrm{Tr}(B)/2)$, which is not significant for our results.

We now invoke a result from \cite[Proof of Theorem~5.2]{Alphonse-20}.

\begin{lemma}\label{lem:eq:Ornstein-Uhlenbeck-singular-space}
	Suppose that $B$ and $Q^{1/2}$ satisfy the Kalman rank condition, i.e., the rank of the matrix
	$(Q^{1/2},BQ^{1/2},\dots,B^{d-1}Q^{1/2})\in\RR^{d\times d^2}$ is $d$, and that there is a set $\cI \subset \{1,\dots,d\}$ such that
	\be\label{eq:powers-R-Ornstein-Uhlenbeck-singular-space}
		\Biggl(\bigcap_{j = 0}^{d-1} \ker(RB^j) \Biggr)^\perp
		=
		\RR_{\cI}^d
		.
	\ee
	Then $S(A)^\perp = \RR_{\cI}^d \times \RR^d$.
\end{lemma}

\begin{remark}[Kalman rank condition]
	If $R=0$, then the Kalman rank condition guarantees that the Ornstein-Uhlenbeck operator is hypoelliptic and
	there is an explicit formula for the semigroup $(\cT(t))_{t\geq 0}$, see \cite{Kolmogoroff-34}.
\end{remark}

Hence for matrices $B$, $Q$, and $R$ as in Lemma~\ref{lem:eq:Ornstein-Uhlenbeck-singular-space},
our spectral inequality for the partial harmonic oscillator is
applicable towards observability and null-controllability of the generalized Ornstein-Uhlenbeck operators.
This covers also situations where the sensor set is not thick.

To make this more precise, we consider some examples.
In all of these we set $d = 2m$, $m\in\NN$, and write $y = (y^{(1)},y^{(2)}) \in \RR^m\times\RR^m$ for $y \in \RR^d = \RR^{2m}$.

\subsubsection{Kolmogorov equation}\label{ssec:Kolmogorov-equation}

Set
\be\label{eq:Kolmogorov-equation}
	Q
	=
	\begin{pmatrix}
		0 & 0 \\
		0 & 2\Id_{m}
	\end{pmatrix}
	,
	\quad
	B
	=
	\begin{pmatrix}
		0 & -\Id_{m} \\
		0 & 0
	\end{pmatrix}
	,
	\quad
	\text{and}
	\quad
	R
	=
	0
	.
\ee
Then $A = \Delta_{x^{(1)}} - x^{(2)}\cdot\nabla_{x^{(1)}}$ is the Kolmogorov operator.
Since
\[
	B^2 = 0
	,
	\quad
	Q^{1/2}
	=
	\begin{pmatrix}
		0 & 0 \\
		0 & \sqrt{2}\Id_{m}
	\end{pmatrix}
	,
	\quad\text{and}\quad
	BQ^{1/2}
	=
	\begin{pmatrix}
		0 & -\sqrt{2}\Id_{m} \\
		0 & 0
	\end{pmatrix}
\]
we have
\[
	(Q^{1/2},BQ^{1/2},\dots,B^{d-1}Q^{1/2})
	=
	\begin{pmatrix}
		0 & 0 & 0 & -\sqrt{2}\Id_{m} & 0 & \dots & 0\\
		0 & \sqrt{2}\Id_{m} & 0 & 0 & 0 & \dots & 0
	\end{pmatrix}
	,
\]
so that $\mathrm{rank}(Q^{1/2},BQ^{1/2},\dots,B^{d-1}Q^{1/2}) = d$.
Thus, the Kalman rank condition is satisfied and \eqref{eq:powers-R-Ornstein-Uhlenbeck-singular-space} holds with $\cI = \emptyset$,
so that  $S(A)^\perp = \{0\}\times\RR^d$.
In particular, we recover the observability of the Kolmogorov equation from thick sensor sets $\omega$, which has already
been obtained in \cite{Alphonse-20}, see also \cite{RousseauM-16, BeauchardPS-18}.

However, if we allow for non-zero $R$, we get an example where our result is strictly stronger than \cite{Alphonse-20}: Let $B$ and $Q$ be
as in \eqref{eq:Kolmogorov-equation}. Since $B^2 = 0$, it is easy to see that the condition
\eqref{eq:powers-R-Ornstein-Uhlenbeck-singular-space} for $R \neq 0$ reduces to
\be\label{eq:reduced-kernel-condition}
	(\mathrm{ker}(R)\cap \mathrm{ker}(RB))^\perp
	=
	\RR_{\cI}^d
	.
\ee
Since for $K,L\in \RR^{m\times m}$ we have
\[
	\begin{pmatrix}
		0 & K \\
		0 & L
	\end{pmatrix}
	\cdot
	\begin{pmatrix}
		0 & -\Id_{m} \\
		0 & 0
	\end{pmatrix}
	=
	0
	,
\]
Condition \eqref{eq:reduced-kernel-condition} further reduces to
\[
	(\RR^m\times (\mathrm{ker}(K)\cap\mathrm{ker}(L)))^\perp
	=
	\mathrm{ker}(R)^\perp
	=	\RR^d_{\cI}
	\quad \text{if} \quad
	R = \begin{pmatrix}
		0 & K \\
		0 & L
	\end{pmatrix}
	.
\]
Hence, if $\{0\}\neq\mathrm{ker}(K)\cap\mathrm{ker}(L)\subset\RR^m_{\cI_2}$ for some $\cI_2\subset\{1,\dots,m\}$,
we have $S(A)^\perp = \RR^d_\cI\times\RR^d$ with $\cI = \{1,\dots,m\}\cup(m+\cI_2)$.

Therefore, our results improve upon \cite{Alphonse-20} since in view of Theorem~\ref{thm:gen}
our sensor sets do not need to be thick.

\subsubsection{Kramers-Fokker-Planck equation}\label{ssec:Kramers-Fokker-Planck-equation}

Let $\cI_1\subset\{1,\dots,m\}$, and set
\bes
	Q
	=
	\begin{pmatrix}
		0 & 0 \\
		0 & 2\Id_{m}
	\end{pmatrix}
	,
	\quad
	B
	=
	\begin{pmatrix}
		0 & \Id_m \\
		-\Id_{\cI_1} & 0
	\end{pmatrix}
	,
	\quad
	\text{and}\quad
	R
	=
	\begin{pmatrix}
		0 & 0 \\
		0 & \frac{1}{2}\Id_m
	\end{pmatrix}
	.
\ees
The operator $A$ is then given by
\[
	A
	=
	\Delta_{x^{(2)}} - \frac{1}{4}|x^{(2)}|^2 - x^{(2)}\cdot\nabla_{x^{(1)}} + \nabla_{x^{(1)}}V(x^{(1)})\cdot\nabla_{x^{(2)}}
	,
\]
where $V(x^{(1)}) = |(x^{(1)})_{\cI_1}|^2$ is the so-called \emph{external potential}.

Note that $B^{2k} = \Id_d$ and $B^{2k+1} = B$ for all $k\in\NN_0$. Since $Q$ is as in Subsection~\ref{ssec:Kolmogorov-equation}
above, we have
\[
	BQ^{1/2}
	=
	\begin{pmatrix}
		0 & \sqrt{2}\Id_m \\
		0 & 0
	\end{pmatrix}
	.
\]
Therefore,
\[
	(Q^{1/2},BQ^{1/2},\dots,B^{d-1}Q^{1/2})
	=
	(Q^{1/2},BQ^{1/2},Q^{1/2},\dots,BQ^{1/2})
\]
has rank $d$ and the Kalman rank condition is satisfied.

On the other hand, the identities for the powers of $B$ imply
\[
	\bigcap_{j = 0}^{d-1} \ker(RB^j)
	=
	\ker(R)\cap\ker(RB)
	.
\]
We have $\ker(R) = \RR^m\times\{0\}$, and since
\[
	RB
	=
	\begin{pmatrix}
		0 & 0 \\
		-\frac{1}{2}\Id_{\cI_1} & 0
	\end{pmatrix}
	,
\]
we calculate $\ker(RB) =\RR^m_{\cI_1^\complement}\times\RR^m$.
Hence,
\[
	\Biggl(\bigcap_{j = 0}^{d-1} \ker(RB^j)\Biggr)^\perp
	=
	\Bigl(\RR^m_{\cI_1^\complement}\times\{0\}\Bigr)^\perp
	=
	\RR^m_{\cI_1}\times\RR^m
	=
	\RR^d_{\cI}
\]
with $\cI = \cI_1\cup\{m+1,\dots,d\}$.
Therefore, by Lemma~\ref{lem:eq:Ornstein-Uhlenbeck-singular-space}, $S(A)^\perp = \RR^d_{\cI}\times\RR^d$.

Since $\cI\neq\emptyset$, our result generalizes those obtained in \cite[Section~4]{Alphonse-20} (and in all previous works) for
all choices of $\cI_1$. Indeed, we have shown that there are non-thick sensor sets $\omega$ from which the abstract Cauchy problem
associated to $A$ is observable, whereas \cite{Alphonse-20} only establishes observability from thick sets.

Note that
\begin{enumerate}[(i)]
	\item if $\cI_1 = \{1,\dots,m\}$ we have $V = |x_1|^2$ and $A$ is called the Kramers-Fokker-Planck operator with quadratic
	external potential. Here the singular space is $S(A) = \{0\}$, and we are in the setting of
	Subsection~\ref{ssec:zero-singular-space} above, that is, we compare $A$ with the (full) harmonic oscillator.

	\item if $\cI_1 \subsetneq \{1,\dots,m\}$, we have a Kramers-Fokker-Planck operator with partial quadratic external potential
	(or without external potential if $\cI_1 = \emptyset$) and we compare $A$ with the partial harmonic oscillator.
\end{enumerate}
%
%
\appendix

\section{The partial harmonic oscillator}\label{sec:partharmOsc}

Let $\cI,\cJ \subset \{ 1,...,d \}$. Consider
\[
	\cH_\cJ
	:=
	\{ f \in L^2(\RR^d) \colon \partial_k f \in L^2(\RR^d)\ \forall k \in \cJ \}
\]
with the norm
\[
	\norm{f}_\cJ
	:=
	\Bigl( \norm{f}_{L^2(\RR^d)}^2 + \sum_{k \in \cJ} \norm{\partial_k f}_{L^2(\RR^d)}^2 \Bigr)^{1/2}
	,\quad
	f \in \cH_\cJ
	.
\]
A standard proof shows that $(\cH_\cJ,\norm{\cdot}_\cJ)$ is complete.

Define the forms
\[
	\aa_\cJ[f , g]
	:=
	\sum_{k \in \cJ} \langle \partial_k f , \partial_k g \rangle_{L^2(\RR^d)}
	,\quad
	f,g \in \cD[\aa_\cJ] := \cH_\cJ
	,
\]
as well as $\cD[\vv_\cI] := \{ f \in L^2(\RR^d) \colon \abs{x_\cI}f \in L^2(\RR^d) \}$,
\[
	\vv_\cI[f , g]
	:=
	\langle \abs{x_\cI}f , \abs{x_\cI}g \rangle_{L^2(\RR^d)}
	,\quad
	f,g \in \cD[\vv_\cI]
	,
\]
and
\[
	\hh_{\cI,\cJ}
	:=
	\aa_\cJ + \vv_\cI
	,\quad
	\cD[\hh_{\cI,\cJ}] := \cD[\aa_\cJ] \cap \cD[\vv_\cI]
	.
\]
The nonnegative form $\aa_\cJ$ is closed since $(\cH_\cJ,\norm{\cdot}_\cJ)$ is complete, and $\vv_\cI$ is nonnegative and closed
by \cite[Proposition~10.5\,(ii)]{Schmuedgen-12}. Thus, the form $\hh_{\cI,\cJ}$ is densely defined, nonnegative, and closed by
\cite[Corollary~10.2]{Schmuedgen-12}, so that there is a unique self-adjoint operator $H_{\cI,\cJ}$ on $L^2(\RR^d)$ given by
\eqs{
	\cD(H_{\cI,\cJ})
	=
	\{ f \in \cD[\hh_{\cI,\cJ}] \colon \exists h &\in L^2(\RR^d) \text{ s.t. }\\
	&\hh_{\cI,\cJ}[f,g] = \langle h , g \rangle_{L^2(\RR^d)}\ \forall g \in \cD[\hh_{\cI,\cJ}] \}
}
and
\[
	\hh_{\cI,\cJ}[f,g]
	=
	\langle H_{\cI,\cJ}f , g \rangle_{L^2(\RR^d)}
	,\quad
	f \in \cD(H_{\cI,\cJ}),\ g \in \cD[\hh_{\cI,\cJ}]
	.
\]

Recall that
\[
	\cG_{\cI,\cJ}
	=
	\{ f \in L^2(\RR^d) \colon x^\alpha \partial^\beta f \in L^2(\RR^d)\ \forall \alpha \in \NN_{0,\cI}^d,\, \beta \in \NN_{0,\cJ}^d \}
	\subset
	\cD[\hh_{\cI,\cJ}]
	.
\]

\begin{lemma}\label{lem:reprHamiltonian}
	We have $\cG_{\cI,\cJ} \subset \cD(H_{\cI,\cJ})$ and $H_{\cI,\cJ}f = -\Delta_\cJ f + \abs{x_\cI}^2 f$ for all $f \in \cG_{\cI,\cJ}$.
	In particular, $\cG_{\cI,\cJ}$ is invariant for $H_{\cI,\cJ}$.
\end{lemma}

\begin{proof}
	Let $f \in \cG_{\cI,\cJ}$ and $g \in \cD[\hh_{\cI,\cJ}] \subset \cH_\cJ$. Then, using Fubini's theorem and that $C_c^\infty(\RR)$
	is dense in $H^1(\RR)$, integration by parts in each coordinate of $\cJ$ yields
	\eqs{
		\hh_{\cI,\cJ}[f,g]
		&=
		\sum_{k \in \cJ} \langle \partial_k f , \partial_k g \rangle_{L^2(\RR^d)} + \langle \abs{x_\cI}f , \abs{x_\cI}g \rangle_{L^2(\RR^d)}\\
		&=
		\langle -\Delta_\cJ f , g \rangle_{L^2(\RR^d)} + \langle \abs{x_\cI}^2 f , g \rangle_{L^2(\RR^d)}\\
		&=
		\langle -\Delta_\cJ f + \abs{x_\cI}^2 f , g \rangle_{L^2(\RR^d)}
		.
	}
	Since $g \in \cD[\hh_{\cI,\cJ}]$ was arbitrary, this proves the claim.
\end{proof}%

We now prove a tensor representation for the operator $H_{\cI,\cJ}$ with $\cJ \neq \emptyset$ and derive related representations
for the elements of spectral subspaces for $H_{\cI,\cJ}$; for a more detailed discussion
of tensor products of operators, we refer to \cite[Section~7.5]{Schmuedgen-12} and \cite[Section~8.5]{Weidmann-80}.

Without loss of generality, we may reorder the coordinates of $\RR^d$ in the following way:
There are $d_1,d_2, d_3 \in \{0,\dots,d\}$ with $1 \leq d_1+d_2 \leq d$ and $d_3 = d - d_1 - d_2$
such that
\be\label{eq:IJspecific}
	\cJ
	=
	\Gamma_1\cup \Gamma_2
	\quad\text{and}\quad
	\cI
	=
	\Gamma_1\cup \Gamma_3
\ee
where
\bes
	\Gamma_1
	=
	\{ 1,\dots,d_1 \},\ \Gamma_2 = \{ d_1+1,\dots,d_1+d_2 \},\ \Gamma_3 = \{ d_1+d_2+1,\dots,d \}
	.
\ees
Analogously to $H_{\cI,\cJ}$ above, we introduce the self-adjoint nonnegative operators $H_1$, $H_2$,
and $H_3$ corresponding to the expressions
\[
	-\Delta + \abs{x}^2 \quad\text{in }L^2(\RR^{d_1}),\quad
	-\Delta \quad\text{in }L^2(\RR^{d_2}),\quad
	\abs{x}^2 \quad\text{in }L^2(\RR^{d_3}),
\]
respectively, via their quadratic forms. Specifically for $H_1$, this form reads
\[
	\hh_1[f,g]
	=
	\sum_{k=1}^{d_1} \langle \partial_k f , \partial_k g \rangle_{L^2(\RR^{d_1})} + \langle \abs{x}f , \abs{x}g \rangle_{L^2(\RR^{d_1})}
\]
for
\[
	f,g \in \cD[\hh_1] := H^1(\RR^{d_1}) \cap \{ h \in L^2(\RR^{d_1}) \colon \abs{x}h \in L^2(\RR^{d_1}) \},
\]
and similarly for $H_2$ and $H_3$.

\begin{lemma}\label{lem:reprH}
	With $\cI$ and $\cJ$ as in \eqref{eq:IJspecific}, the operator $H = H_{\cI,\cJ}$ admits the tensor representation
	\be\label{eq:reprH}
		H
		=
		H_1 \otimes I_2 \otimes I_3 + I_1 \otimes H_2 \otimes I_3 + I_1 \otimes I_2 \otimes H_3
		,
	\ee
	where $I_j$ denotes the identity operator in $L^2(\RR^{d_j})$, $j = 1,2,3$.
\end{lemma}

\begin{proof}
	Denote the operator corresponding to the right-hand side of \eqref{eq:reprH} by $\tilde{H}$. Following
	\cite[Theorem~7.23 and Exercise~7.17.a.]{Schmuedgen-12}, $\tilde{H}$ is nonnegative and self-adjoint with operator core
	$\cD := \Span_\CC\{f_1 \otimes f_2 \otimes f_3 \colon f_j \in \cD(H_j)\}$. Moreover, using the form domains of
	$H_j$, it is easy to see that $\cD \subset \cD[\hh_{\cI,\cJ}]$. We now proceed similarly as in ~\cite[Section~3]{Seelmann-21}:
	Let $f = f_1 \otimes f_2 \otimes f_3 \in \cD$ and $g \in \cD[\hh_{\cI,\cJ}]$. By Fubini's theorem we then have that
	$g(\cdot,y,z)$ belongs for almost every $(y,z) \in \RR^{d_2} \times \RR^{d_3}$ to the form domain $\cD[\hh_1]$ of $H_1$. Using
	this, we see that
	\eqs{
		\langle (H_1 \otimes I_2 \otimes I_3)f , g &\rangle_{L^2(\RR^d)}
		=
		\langle (H_1f_1 \otimes f_2 \otimes f_3) , g \rangle_{L^2(\RR^d)}\\
		&=
		\int_{\RR^{d_2}\times\RR^{d_3}} f_2(y)f_3(z) \langle H_1f , g(\cdot,y,z) \rangle_{L^2(\RR^{d_1})} \Diff(y,z)\\
		&=
		\int_{\RR^{d_2}\times\RR^{d_3}} f_2(y)f_3(z) \hh_1[f_1 , g(\cdot,y,z)] \Diff(y,z)\\
		&=
		\hh_{\Gamma_1,\Gamma_1}[f , g]
		.
	}
	In a completely analogous way, we establish
	\[
		\langle (I_1 \otimes H_2 \otimes I_3)f , g \rangle_{L^2(\RR^d)}
		=
		\hh_{\emptyset,\Gamma_2}[f , g]
	\]
	and
	\[
		\langle (I_1 \otimes I_2 \otimes H_3)f , g \rangle_{L^2(\RR^d)}
		=
		\hh_{\Gamma_3,\emptyset}[f , g]
		.
	\]
	Summing up gives
	\[
		\langle \tilde{H}f , g \rangle_{L^2(\RR^d)}
		=
		\hh_{\Gamma_1,\Gamma_1}[f , g] + \hh_{\emptyset,\Gamma_2}[f , g] + \hh_{\Gamma_3,\emptyset}[f , g]
		=
		\hh_{\cI,\cJ}[f,g]
		.
	\]
	By sesquilinearity, the latter extends to all $f \in \cD$, so that $\tilde{H}|_{\cD} \subset H$. Since $\cD$ is an
	operator core for $\tilde{H}$ and both $H$ and $\tilde{H}$ are self-adjoint, we conclude that
	$\tilde{H} = \overline{\tilde{H}|_{\cD}} = H$, which proves the claim.
\end{proof}%

\begin{corollary}\label{cor:specH}
	In the situation of Lemma~\ref{lem:reprH}, we have
	\[
		\sigma(H)
		=
		\sigma(H_1) + \sigma(H_2) + \sigma(H_3)
		,
	\]
	and the restriction of $H$ to the Schwartz functions on $\RR^d$ is essentially
	self-adjoint.
\end{corollary}

\begin{proof}
	The first part follows from \cite[Corollary~7.25]{Schmuedgen-12}; cf.\ also \cite[Exercise 7.18.a.]{Schmuedgen-12}.
	For the second part we observe that $H_2$ and $H_3$ are essentially self-adjoint on $\cS(\RR^{d_2})$ and $\cS(\RR^{d_3})$, respectively.
	Moreover, since the eigenfunctions of $H_1$ are Hermite functions and, in particular, Schwartz functions, $H_1$ is also
	essentially self-adjoint on $\cS(\RR^{d_1})$ by Nussbaum's theorem, see, e.g., \cite[Theorem~7.14]{Schmuedgen-12}.
	Hence, $H$ is essentially self-adjoint on $\Span_\CC\{f_1\otimes f_2\otimes f_3\colon f_j\in\cS(\RR^{d_j})\}\subset\cS(\RR^d)$, see,
	e.g., \cite[Theorem~8.33]{Weidmann-80}.
\end{proof}%

\begin{remark}\label{rem:truePartHarm}
	If $d_3 = 0$, then the third tensor factor can be dropped here, that is, we then have
	$H = H_1 \otimes I_2 + I_1 \otimes H_2$ and $\sigma(H) = \sigma(H_1) + \sigma(H_2)$.
\end{remark}

Since $H_1$ has pure point spectrum, in the situation of the preceding remark we obtain the following result.

\begin{corollary}\label{cor:specfamH}
	In the situation of Lemma~\ref{lem:reprH} with $d_3 = 0$,
	every $f \in \Ran P_{(-\infty,\lambda]}(H)$, $\lambda \geq 0$,
	can be represented as a finite sum
	\[
		f
		=
		\sum_k \phi_k \otimes \psi_k
	\]
	with suitable $\phi_k \in \Ran P_{(-\infty,\lambda]}(H_1)$ and $\psi_k \in \Ran P_{(-\infty,\lambda]}(H_2)$.
	In particular, $f$ can be extended to an analytic function on $\CC^d$.
	Moreover, $(\partial^\beta f)(\cdot,y)$ belongs to $\Ran P_{(-\infty,\lambda]}(H_1)$ for all $y \in \RR^{d_2}$ and all multiindices
	$\beta \in \NN_{0,\cI^\complement}^d$, and $(\partial^\beta f)(x,\cdot)$ belongs to $\Ran P_{(-\infty,\lambda]}(H_2)$ for all
	$x \in \RR^{d_1}$ and all $\beta \in \NN_{0,\cI}^d$.
\end{corollary}

\begin{proof}
	We proceed similarly as in the proof of \cite[Lemma~2.3]{EgidiS-21}.
	Let $f \in \Ran P_{(-\infty,\lambda]}(H)$ for some $\lambda \geq 0$, and let $(\phi_k)_k$ be an orthonormal basis of eigenfunctions
	of $H_1$ with corresponding eigenvalues $\mu_k$. Write
	$f(x,y) = \sum_k \langle f(\cdot,y) , \phi_k \rangle_{L^2(\RR^{d_1})} \phi_k(x) = \sum_k \phi_k(x) g_k(y)$, where
	$g_k \in L^2(\RR^{d_2})$ is given by $g_k(y) = \langle f(\cdot,y) , \phi_k \rangle_{L^2(\RR^{d_1})}$.
	
	By \cite[Theorem~8.34 and Exercise~8.21]{Weidmann-80},
	the spectral family $P_{(-\infty,\lambda]}(H)$ for $H$
	admits the representation
	\[
		P_{(-\infty,\lambda]}(H)
		=
		\sum_{\mu \leq \lambda} P_{\{\mu\}}(H_1) \otimes P_{(-\infty,\lambda-\mu]}(H_2)
		.
	\]
	This implies that
	\[
		f
		=
		P_{(-\infty,\lambda]}(H)f
		=
		\sum_{k\colon \mu_k \leq \lambda} \phi_k \otimes \psi_k
	\]
	with $\psi_k = P_{(-\infty,\lambda-\mu_k]}(H_2)g_k \in \Ran P_{(-\infty,\lambda-\mu_k]}(H_2) \subset \Ran P_{(-\infty,\lambda]}(H_2)$.
	This shows the first part of the statement. The remaining part is then clear by using the corresponding properties of $\phi_k$
	and $\psi_k$ and Hartogs' theorem on separate analyticity.
\end{proof}%

We close this appendix by showing that for general $\cI,\cJ \subset \{1,\dots,d\}$ we can trade the parts of the potential corresponding
to elements in $\cI\setminus\cJ$ for additional derivatives via an appropriate partial Fourier transform. Let $m = \#(\cI\setminus\cJ)$,
and write
\[
	x
	=
	(x^{(1)},x^{(2)})
	\quad \text{with}\quad
	x^{(1)}
	\in
	\RR^d_{\cI\setminus\cJ}
	,
	\quad
	x^{(2)}
	\in
	\RR^d_{(\cI\setminus\cJ)^\complement}
	.
\]
Consider the partial Fourier transform
\be\label{eq:partFourierTransform}
	(\cF_{\cI\setminus\cJ} f)(x)
	=
	\frac{1}{(2\pi)^{m/2}}
	\int_{\RR^d_{\cI\setminus\cJ}} f(\eta,x^{(2)}) \euler^{-i \eta\cdot x^{(1)}}\Diff{\eta}
	,\quad
	f \in L^2(\RR^d)
	,
\ee
which by Plancherel's and Fubini's theorems defines a unitary operator on $L^2(\RR^d)$. With this transform, we can now show that
$H_{\cI,\cJ}$ is unitarily equivalent to a partial harmonic oscillator $H_{\cI',\cJ'}$ with $\cJ' \supset \cI'$.

\begin{lemma}\label{lem:operators-and-partial-ft}
	With $\cF_{\cI\setminus\cJ}$ as in \eqref{eq:partFourierTransform} we have
	\[
		H_{\cI,\cJ}
		=
		\cF_{\cI\setminus\cJ}^{-1}
		H_{\cI\cap\cJ,\cI\cup\cJ}
		\cF_{\cI\setminus\cJ}
		.
	\]
	In particular, for all $\lambda \geq 0$ we have
	\[
		\cF_{\cI\setminus\cJ}\Ran P_{(-\infty,\lambda]}(H_{\cI,\cJ})
		=
		\Ran P_{(-\infty,\lambda]}(H_{\cI\cap\cJ,\cI\cup\cJ})
		.
	\]
\end{lemma}

\begin{proof}
	We first observe that for $j \in \cJ$ the partial derivative $\partial_j$ commutes with $\cF_{\cI\setminus\cJ}$, and that for
	$j \in \cI \cap \cJ$ the multiplication by $x_j$ commutes with $\cF_{\cI\setminus\cJ}$. Moreover, for $f \in L^2(\RR^d)$ and
	$k \in \cI \setminus \cJ$ we have $x_k f \in L^2(\RR^d)$ if and only if $\partial_k\cF_{\cI\setminus\cJ} f \in L^2(\RR^d)$ and,
	in this case, $\partial_k \cF_{\cI\setminus\cJ} f = -\ii\cF_{\cI\setminus\cJ}(x_kf)$. With this, we easily see that
	$f \in \cD[\hh_{\cI,\cJ}]$ if and only if $\cF_{\cI\setminus\cJ} f \in \cD[\hh_{\cI\cap\cJ,\cI\cup\cJ}]$ and that
	\[
		\hh_{\cI\cap\cJ,\cI\cup\cJ}[\cF_{\cI\setminus\cJ} f,\cF_{\cI\setminus\cJ} g]
		=
		\hh_{\cI,\cJ}[f,g]
		\quad
		\text{for all}
		\quad
		f,g
		\in
		\cD[\hh_{\cI,\cJ}]
		,
	\]
	which proves the claim.
\end{proof}%

\newcommand{\etalchar}[1]{$^{#1}$}

\end{document}